\documentclass[reqno,11pt,centertags]{article}
\usepackage{amsmath,amsthm,amscd,amssymb,latexsym,upref,bm}
\date{\today}

\input epsf
\usepackage{epsfig}
\usepackage[T2A,OT1]{fontenc}
\usepackage[ot2enc]{inputenc}
\usepackage[russian,english]{babel}
\usepackage[margin=3.5 cm]{geometry}
\usepackage{epic,eepicemu}
\usepackage[all]{xy}
\usepackage{xcolor}

\newcommand{\bbE}{{\mathbb{E}}}
\newcommand{\bbD}{{\mathbb{D}}}
\newcommand{\bbN}{{\mathbb{N}}}
\newcommand{\bbR}{{\mathbb{R}}}

\newcommand{\bbZ}{{\mathbb{Z}}}
\newcommand{\bbC}{{\mathbb{C}}}

\DeclareMathAlphabet{\mathpzc}{OT1}{pzc}{m}{it}

\newcommand{\cA}{{\mathcal{A}}}
\newcommand{\cB}{{\mathcal{B}}}

\newcommand{\cD}{{\mathcal{D}}}
\newcommand{\cE}{{\mathcal{E}}}
\newcommand{\cF}{{\mathcal{F}}}

\newcommand{\cH}{{\mathcal{H}}}

\newcommand{\cK}{{\mathcal{K}}}

\newcommand{\cN}{{\mathcal{N}}}

\newcommand{\cR}{{\mathcal{R}}}
\newcommand{\cS}{{\mathcal{S}}}
\newcommand{\cV}{{\mathcal{V}}}
\newcommand{\cU}{{\mathcal{U}}}
\newcommand{\cW}{{\mathcal{W}}}

\newcommand{\ff}{{\mathfrak{f}}}
\newcommand{\fj}{{\mathfrak{j}}}

\newcommand{\fw}{{\mathfrak{w}}}

\newcommand{\bTh}{{\mathbf{\Theta}}}

\newcommand{\bs}{{\mathbf{s}}}

\newcommand{\bl}{{\bm{\lambda}}}

\newcommand{\e}{{\epsilon}}
\renewcommand{\k}{\varkappa}
\newcommand{\vt}{\vartheta}

\newcommand{\pd}{{\partial}}

\newcommand{\un}{\underline}

\def\restr#1{\,\vrule\,\lower1ex\hbox{$#1$}}

\def\a{\alpha}
\def\b{\beta}
\def\d{\delta}

\def\e{\epsilon}
\def\ve{\varepsilon}

\def\g{\gamma}
\def\G{\Gamma}

\def\k{\kappa}
\def\vk{\varkappa}
\def\l{\lambda}
\def\u{\upsilon}

\def\O{\Omega}

\def\s{\sigma}

\def\t{\theta}
\def\T{\Theta}
\def\x{\xi}
\def\vt{\vartheta}

\newcommand{\dom}{\operatorname{dom}}
\newcommand{\KdV}{\operatorname{KdV}}

\newcommand{\dd}{\mathrm{d}}

\newcommand{\oc}{\overset{\circ}}

\renewcommand{\Re}{\text{\rm Re}\,}
\renewcommand{\Im}{\text{\rm Im}\,}

\newcommand{\clos}{\text{\rm clos}}

\newcommand{\Res}{\text{\rm Res}\,}

\allowdisplaybreaks \numberwithin{equation}{section}
\newtheorem{theorem}{Theorem}[section]
\newtheorem*{theorem*}{Theorem}

\newtheorem{lemma}[theorem]{Lemma}
\newtheorem{proposition}[theorem]{Proposition}

\newtheorem{corollary}[theorem]{Corollary}

\theoremstyle{definition}
\newtheorem{definition}[theorem]{Definition}

\newtheorem{remark}[theorem]{Remark}

\date{\today}

\author{B. Eichinger\thanks{Supported by the Austrian Science Fund FWF, project no: J 4138-N32.}, T. VandenBoom \thanks{Supported in part by NSF grant DMS-1148609.}, P. Yuditskii\thanks{Supported by the Austrian Science Fund FWF, project no: P 29363-N32.}}

\title{KdV hierarchy via Abelian coverings and operator identities
}

\begin{document}

\maketitle

\begin{abstract}
We establish precise spectral criteria for potential functions $V$ of reflectionless Schr\"odinger operators $L_V = -\partial_x^2 + V$ to admit solutions to the Korteweg de-Vries (KdV) hierarchy with $V$ as an initial value.  More generally, our methods extend the classical study of algebro-geometric solutions for the KdV hierarchy to noncompact Riemann surfaces by defining generalized Abelian integrals and analogues of the Baker-Akhiezer function on infinitely connected domains with a uniformly thick boundary satisfying a fractional moment condition.
\end{abstract}

\section{Introduction}

We study the Cauchy problem for the Korteweg de-Vries (KdV) equation
\begin{align*}
\partial_t V &= \frac{1}{4}\partial_x^3 V - \frac{3}{2}V\partial_x V \\
V(\cdot, 0) &= V(\cdot)
\end{align*}
and the associated hierarchy of higher-order differential equations
\begin{align}
\label{eq:kdvh.1}
\partial_{t_k} V &= \KdV_k(V) \\
\label{eq:kdvh.2}
V(\cdot,0) &= V(\cdot)
\end{align}
for $k \in \bbN$, which we will define precisely below.  For the sake of introduction, we simply note that $\KdV_1(V) = \frac{1}{4}\partial_x^3 V - \frac{3}{2}V\partial_x V$ is the typical KdV equation, and that $\KdV_k$ is generally an order $2k+1$ polynomial differential operator.  We establish conditions on $V$ in terms of the \textit{spectral properties} of the associated Schr\"odinger operator $L_V = -\partial_x^2 + V$ which guarantee global existence and temporal and spatial almost-periodicity of classical solutions:
\begin{theorem}
\label{t:introthm}
Assume that the family of operators $\{L_{V_{x_0}} : V_{x_0}(x) = V(x+ x_0), x_0 \in \bbR\}$ is ergodic.  Denote by $E$ their common spectral set, $E_{ac}$ their almost-sure a.c. spectrum, and by $\rho(\x)/\sqrt{\x}$ the associated integral density of states (i.d.s).  

Suppose that $E$ is homogeneous, $E = E_{ac}$, and $E$ satisfies the following moment condition:
\begin{align*}
\int_{\bbR_+ \setminus E} \dd \x^{k+\frac{3}{2}} < \infty
\end{align*}
and that the i.d.s. satisfies the entropy condition
\begin{align}\label{eq:entropycondn}
-\int_E \log\rho(\x) \frac{\rho(\x)\dd \x}{\sqrt{\x}} < \infty.
\end{align}
Then the Cauchy problem \eqref{eq:kdvh.1}, \eqref{eq:kdvh.2} admits a global solution $V(x, t_k)$ which is uniformly almost-periodic in the time and space coordinates.
\end{theorem}

We present this result first because of its broad appeal.  Recent work \cite{BDGL} of Binder, Damanik, Goldstein, and Lukic has, under stronger assumptions, proven something similar in the case $k = 1$, and Kotani has announced a related result under an integer moment condition \cite{Kot}.  In our case, Theorem \ref{t:introthm} is only a facet of our main theorem, Theorem \ref{t:mainthm}.  In fact, the full content of this paper is the development of a spectrally-dependent Fourier transform, with respect to which the study of reflectionless Schr\"odinger operators becomes greatly simplified.

Naturally, our result is based on Lax pair representation in the theory of integrable systems, see e.g. \cite{DKN, GeHo03}, the spectral theory of ergodic 1-D Schr\"odinger operators, see e.g. \cite{PaFi92}, and the  functional model approach to the spectral theory see e.g. \cite{Nik}.  The relationship between the KdV equation and the Schr\"odinger operator via the Lax pair formalism was noted in the 1960s by Lax \cite{Lax}.   At the same time, use of inverse scattering techniques to solve the KdV equation was pioneered by Gardner, Greene, Kruskal, and Miura \cite{GGKM}.  The 1970s saw significant further development expounding on these ideas to solve the KdV equation for periodic initial data \cite{Dub,  McKTrub, Nov}.  Shortly thereafter, algebro-geometric extensions of the techniques from the periodic setting were developed to address almost-periodic and ergodic initial conditions having finite-gap spectra; see \cite{GeHo03} for a textbook treatment on this approach.  Extensions of these techniques to the infinite-gap setting have been partially developed in important work by Egorova \cite{Egor}, and more recently in work of Binder, Damanik, Goldstein, and Lukic \cite{BDGL}.  Some of our methods, including the ideas of generalized Abelian integrals and conformal mappings onto comb domains, were established by Marchenko in \cite{Marc}.  Other important methods in integrable systems were developed by Deift and Zhao \cite{DeZhao}.

The sequel will be structured as follows: the remainder of Section 1 establishes definitions and notation en route to our main theorem, Theorem \ref{t:mainthm}.  Section 2 recalls some basic preliminaries from spectral theory, and establishes the spectral infemum as a normalization point.  Section 3 describes in detail the functional models and proves an important identity regarding the reproducing kernels.  Finally, Section 4 discusses the generalized Abelian integrals and their relationship to the finite-gap case, and proves the asymptotic expansion of the $m$-functions up to a certain order under our conditions.

\subsection{Generalized Abelian integrals and Hardy spaces \\ on infinitely-connected domains}

Algebro-geometric solutions for the KdV hierarchy equations are given by means of the Baker-Akhiezer function, see e.g. \cite{Kri78,TSH},
and books \cite{DKN,MumTata2}, see also \cite{GeHo03}. This function is associated to a hyperelliptic  Riemann surface
(with the standard compactification)
\begin{align*}
\cS:= \left\{p = (\lambda, w) : w^2 = \l\prod_{j = 1}^N (\lambda - a_j)(\lambda - b_j)\right\}\cup\{\infty\},
\end{align*}
 and a specified point $\infty\in\cS$ on it. The Baker-Akhiezer function contains two factors: the first one is given by an Abelian integral with a certain singularity at $\infty$, and the second one is represented as a specific ratio of theta-functions closely related to the so-called prime form \cite[IIIb, \textsection 1]{MumTata2}.

Our Baker-Akhiezer function is
  \eqref{19jan50}.
It allows for the much more general case when $N=\infty$, although we are restricted by the assumption that   $ a_j,b_j$ are real. In our normalization $0<a_j<b_j$ and $E:=\bbR_+\setminus \cup_{j\ge 1}(a_j,b_j)$ does not contain any isolated points.  So, we consider the domain
 $$
 \cS_+=\bbC\setminus E, \quad 
 $$
 which would play a role of the ``upper sheet"   for a compact hyperelliptic surface $\cS$, $N<\infty$.
A classical construction of Abelian integrals is given by means of potential theory \cite[Part III, Ch. 9]{HC}.  Our first assumption is that our domain is standard in this regard:
 \begin{itemize}
 \item[(R)] Let $G_{\l_0}(\l)=G(\l,\l_0)$ be the Green function in the domain $\cS_+$ with logarithmic singularity at $\l_0\in\cS_+$. We assume that
 $G_{\l_0}$ is continuous up to the boundary,
that is, $\cS_+$ is regular in terms of the potential theory,   \cite[Theorem 6.3., p. 95]{GM}.
\end{itemize}
An interpretation of an Abelian integral of the third kind in terms of the Green function is well known; specifically, the generalized Abelian integral of the third kind $-\log \Phi(\l,\l_0)$ is related to the Green function by
\begin{equation}\label{18jan10}
|\Phi(\l,\l_0)|=e^{-G(\l,\l_0)}, \quad \l_0\in \cS_+.
\end{equation}
Letting $\pi_1(\cS_+)$ be the fundamental group of the given domain, analytic continuation of the function $\Phi_{\l_0}(\l)=\Phi(\l,\l_0)$ along the path $\gamma \in \pi_1(\cS_+)$ picks up a unimodular mutliplicative factor
\begin{equation}\label{18jan11}
\Phi(\g(\l),\l_0)=e^{2\pi i \nu_{\l_0}(\g)}\Phi(\l,\l_0), \quad \g\in\pi_1(\cS_+),\quad \nu_{\l_0}(\g)\in\bbR/\bbZ.
\end{equation}
In this case we say $\Phi_{\l_0}(\l)$ is (multiplicative) character automorphic with character $\nu_{\l_0}$.

The counterpart of the theta-function is hardly possible in our level of generality, but the ratio of two of them can have perfect sense; we suggest to treat it as a \textit{special function} associated to the problem.  We interpret the prime form (the second factor in our generalized Baker-Akhiezer function) as the reproducing kernel of a suitable Hilbert space of analytic functions.  By ``suitable Hilbert spaces" we mean Hardy spaces $H^2_{\cS_+}(\a)$ associated to an arbitrary unitary character $\a$ of the group $\pi_1(\cS_+)$:
\begin{definition}\label{d18jan1}
Let $e^{2\pi i\a(\g)}$ be the system of multipliers associated  to paths $\g\in\pi_1(\cS_+)$. We say that $f(\l)$ belongs to the space  
$H^2_{\cS_+}(\a)$ if
\begin{itemize}
\item
it is a locally analytic multivalued function in $\cS_+$;
\item $|f(\l)|^2$ is single-valued in $\cS_+$ and  possesses a harmonic majorant, i.e. there exists a function $u(\l)$ harmonic in $\cS_+$  such that
$$
|f(\l)|^2\le u(\l);
$$
\item its analytic continuation  $f(\g(\l))$ along $\g\in\pi_1(\cS_+)$ is related to the original value $f(\l)$ by 
$$
f(\g(\l))=e^{2\pi i\a(\g)}f(\l).
$$
\end{itemize}
  
We define
\begin{equation}\label{18jan0}
\|f\|^2_{H^2_{\cS_+}(\a)}=\inf\{u(-1): \Delta u=0,\quad |f(\l)|^2\le u(\l),\ \l\in\cS_+\}.
\end{equation}
\end{definition}
This space possesses a reproducing kernel, which we denote by $k^\a_{\cS_+}(\l,\l_0)=k^\a_{\cS_+,\l_0}(\l)$. That is,
\begin{equation}\label{18jan1}
\langle f, k^\a_{\cS_+,\l_0} \rangle= f(\l_0),\quad\forall f\in H^2_{\cS_+}(\a).
\end{equation}
Since the elements of $H^2_{\cS_+}(\a)$ are multivalued functions, \eqref{18jan1} creates a state of small uncertainty. To avoid this uncertainty, we provide an alternative definition using the universal covering for $\cS_+$.

Recall that for the given domain there exists a function $\bl(z)$  analytic  in the upper half-plane $\bbC_+$ and a Fuchsian group $\G \simeq \pi_1(\cS_+)$ such that $\bl(z)$ induces a one to one correspondence between the points $\l\in \cS_+$ and orbits $\{\g(z)\}_{\g\in \G}$, $z\in \bbC_+$. We will normalize this function by the conditions
$$
\bl(iy)\in\bbR_-,\ y>0,\quad \bl(iy)\to -\infty\ \text{as}\ y\to\infty.
$$
Due to this normalization, $\bl(z)$ is defined up to a positive multiplier.  By $\G^*$ we denote the group of unitary characters of the group $\G$.

The following definition is basically parallel to Definition \ref{d18jan1}.
\begin{definition}\label{d18jan2} Let $\a\in\G^*$.
The space $H_{\G}^2(\a)$ is the subspace of the standard $H^2$ in $\bbC_+$ (with respect to the harmonic measure $\frac 1 \pi\frac{dx}{1+x^2}$) consisting of character-automorphic functions, i.e.,
$$
H_{\G}^2(\a)=\{ f\in H^2:\ f(\gamma(z))=e^{2\pi i\a(\g)} f(z)\}.
$$
\end{definition}
The spaces $H^2_{\cS_+}(\a)$ and $H^2_\G(\a)$ correspond by way of the uniformization $\bl(z)$:
\begin{proposition}\label{pr:h2corresp}
$f(\l)$ belongs to $H^2_{\cS_+}(\a)$ if and only if there exists $g\in H_{\G}^2(\a)$ such that $f(\bl(z))=g(z)$.
\end{proposition}
We denote the reproducing kernel in $H_{\G}^2(\a)$ by $k_{\G}^\a(z,z_0)$ and remark that the correspondence in Proposition \ref{pr:h2corresp} completely removes the above-mentioned ambiguity in \eqref{18jan1}.  For a textbook discussion of character automorphic Hardy spaces, see \cite{Has83}.

Our second assumption on $\cS_+$ is the Widom condition, which is necessary and sufficient  for non-triviality of the spaces $H_{\G}^2(\a)$ for all $\a\in\G^*$.

\begin{itemize}
\item[(PW)] We assume that for some (and hence all) $\l_0 \in \cS_+$, we have
\begin{equation}\label{18jan4}
\sum_{\nabla G_{\l_0}(\x)=0} G(\x,\l_0)<\infty.
\end{equation}
\end{itemize}
As soon as \eqref{18jan4} holds, $\cS_+$ is called a domain of Widom type, and $\G$ is respectively called a group of Widom type. 
For $\l_0\in\bbR_-$ all critical values $\x=\x(\l_0)$ are real; moreover, there is exactly one critical point in each gap, $\x_j(\l_0)\in(a_j,b_j)$.

The Widom condition (PW) along with our third assumption allows us to extend the notion of Abelian integrals of the second kind:
\begin{itemize}
\item[($k$-GLC)] By the order $k$ gap length condition we mean
\begin{equation}\label{21nov0}
\int_{\bbR_+\setminus E}\dd \xi^{k+\frac 1 2} =\sum_{j\ge 1}(b_j^{k+\frac 1 2}-a_j^{k+\frac 1 2})<\infty.
\end{equation}
\end{itemize}
Under the assumptions (PW) and ($k$-GLC), we can define the following (see Proposition \ref{pr:genabint} below):
\begin{definition} By the (normalized) order $k$ generalized Abelian integral we mean the multivalued in $\cS_+$ analytic function $\Theta^{(k)}(\l)$, whose
(single-valued) imaginary part is given by 
\begin{equation}\label{18jan5}
\Im \Theta^{(k)}(\l)=M^{(k)}(\l):=\Im \l^{k+\frac 1 2}+\frac 1 \pi\int_{\bbR_+\setminus E} G(\xi,\l) \dd\xi^{k+\frac 1 2}.
\end{equation}
\end{definition}

As before, the analytic continuation  $\Theta^{(k)}(\g(\l))$ along the path $\g\in\pi_1(\cS_+)$ is related to the original value $\Theta^{(k)}(\l)$ by
\begin{equation}\label{18jan6}
\Theta^{(k)}(\g(\l))=\Theta^{(k)}(\l)+2\pi \eta^{(k)}(\g).
\end{equation}
We say that $\eta^{(k)}$ is an additive character on $\pi_1(\cS_+)$, and that the function $\Theta^{(k)}(\l)$ is (additive) character automorphic.
\begin{lemma} The generalized Abelian differentials $\dd\Theta^{(k)}(\l)$ are of the form
\begin{equation}\label{}
\dd \Theta^{(k)}(\l)=\pi^{(k)}(\l)\prod_{j\ge 1}\frac{1-c_j^{(k)}/\l}{\sqrt{(1-a_j/\l)(1-b_j/\l)}}\frac {\dd\l}{\sqrt{\l}}
\end{equation}
where $\pi^{(k)}(\l)$ is a certain monic polynomial of degree $k$ and $c^{(k)}_j\in (a_j,b_j)$.
\end{lemma}
In the special case $k=0$ we will drop the index. $M(\l)$ is called the \textit{Martin function}  of the given domain (with respect to infinity) or 
\textit{Phragm\'en-Lindel\"of function} (particularly, with respect to the representation \eqref{18jan5}, see \cite[Theorem, p. 407]{KoosisLogInt1}).  

\subsection{Functional models on the universal Abelian cover in application to spectral theory of 1-D Schr\"odinger operators}

We can now trace the path towards our Main Theorem. The universal Abelian covering of $\cS_+$ is defined as follows: denote by $\G'$ the commutator subgroup of $\G$
\begin{equation}\label{19jan1}
\G'=\{\oc\g\in\G:\ e^{2\pi i\a(\oc\g)}=1, \ \forall \a\in\G^*\}.
\end{equation}
The quotient $\G/\G'$ is canonically dual to $\G^*$ by Pontryagin duality.  Thus, any of the multi-valued analytic continuations defined above become functions on the surface $\cR:=\bbC_+/\G'$. In particular, we have the covering map $\l_{\cR}:\cR\to\cS_+$
\begin{equation}\label{19jan2}
 \l_{\cR}(p)=\bl(z),\ \text{where}\ p=\{\oc\g(z)\}_{\oc\g\in\G'}\in\cR.
\end{equation}
Similarly, we consider generalized Abelian integrals as functions on  $\cR$ and denote them by $\t_{\cR}^{(k)}(p)=\Theta^{(k)}(\l_{\cR}(p))$. As one would expect, the group $\G/\G'$ acts on this surface by
$$
\un \g p=\{\oc\g(\g(z))\}_{\oc\g\in\G'},\quad p=\{\oc\g(z)\}_{\oc\g\in\G'},
$$
where $\un \g\in \G/\G'$ denotes the equivalence class in $\G/\G'$ of an element $\g\in\G$. According to this notation \eqref{18jan6} has the form
\begin{equation}\label{19jan3}
\t_{\cR}^{(k)}(\un\g p)=\t_{\cR}^{(k)}(p)+2\pi\eta^{(k)}(\g).
\end{equation}

The main theorem claims that the KdV hierarchy equation of order $k$ is simply the relation that the two \textit{multiplication operators} for the functions $\l_{\cR}(p)$ and $\t^{(k)}_{\cR}(p)$ \textit{commute} as actions on a Hardy space $H^2_{\cR}$ associated to the Abelian cover $\cR$; we define this space below.

We require such notions as inner and outer functions, functions of bounded characteristic (of class $N$) and of its Smirnov subclass (or Nevanlinna class $\cN_+$) on the Riemann surfaces. These objects are well known in the theory of Hardy spaces in the disc $\bbD$ or half-plane $\bbC_+$ \cite{Gar07}, particularly the class $\cN_+$ is defined in Ch. II, Sect. 5. We say that the function is outer (inner) on $\cS_+$ or $\cR$ if its lift to the universal cover $\bbC_+$ is outer (inner). We say that $F(p)$, $p\in\cR$, is of bounded characteristic if it can be  represented as a ratio two bounded analytic functions. It is of Smirnov class if in addition the denominator is an outer function.

\begin{definition}
\label{d:h2r}
	The space $H^2_{\cR}$ is formed by Smirnov class functions $F$ on $\cR$  with square-integrable boundary values
	\begin{equation}\label{19jan8}
 \|F\|^2= \frac 1 {2\pi}\int_{\pd\cR}|F(p)|^2\dd\theta_\cR(p)<\infty.
\end{equation}
\end{definition}
The integral density of states is a fundamental measure in the spectral theory of ergodic operators. In our case it coincides with $\dd\Theta$ restricted on $E = \pd\cS_+$.
Above we paid  tribute to its importance in our definition of $H_{\cR}^2$ as a subspace of $L^2$-space with respect to $\dd\t_{\cR}$ on the boundary $\pd\cR$. 

The action of the group $\G/\G'$  on $\pd\cR$ is dissipative. This means that there exists a fundamental measurable set $\bbE\subset \pd\cR$ such that $\un\g \bbE\cap\bbE=\emptyset$ for $\un\g\not=1_{\G/\G'}$, and, for an arbitrary $L^1$-function $F(p)$,
$$
\sum_{\un \g\in\G/\G'}\int_{\un\g \bbE} F(p)\dd\t_{\cR}=\int_{\pd\cR} F(p)\dd\t_{\cR}.
$$
Note that essentially $\bbE=\pd\cS_+$.  We  define the Hardy spaces $H^2(\a)$ of character automorphic functions (with respect to the action of the group $\G/\G'$ on $\cR$) in the following way.
\begin{definition}
\label{d:h2a}
	Let $\a\in\G^*$. The space $H^2(\a)$ consists of those analytic functions $f$ on $\cR$ which satisfy
	\begin{itemize}
		\item[(i)] $f$ is of Smirnov class,
		\item[(ii)] $f(\un\g p)=e^{2\pi i \a(\g)}f(p)$ for all $\un\g\in\G/\G'$,
		\item[(iii)] $\|f\|^2={\frac{1}{2\pi}\displaystyle \int_{\bbE}|f(p)|^2\dd \t_{\cR}(p)<\infty}$.
	\end{itemize}
\end{definition}
The space $H^2(\a)$ is a small modification of the spaces $H_{\cS_+}^2(\a) \cong H^2_\G(\a)$ defined above -- for a precise relationship between these spaces, see Lemma  \ref{l19jan2}.  The advantage of considering these spaces in this new way is that we can collect all character automorphic Hardy spaces $H^2(\a)$ within $H^2_\cR$ in a sense we now describe.

Consider the collection of functions $f(p,\a)$ such that
$f(p,\a)\in H^2(\a)$ for a.e. $\a\in\G^*$ with respect to the Haar measure $\dd\a$ on $\G^*$ and
$$
\| f\|^2:=\int_{\G^*}\|f(p,\a)\|_{H^2(\a)}^2 \dd\a<\infty.
$$
We denote this space by $\cH=\int_{\G^*} H^2(\a) \dd\a$.


The proof of the main theorem concerns two natural Fourier transforms which unify the perspectives of the global functional model on $H^2_\cR$ and the individual functional models on the character automorphic Hardy spaces $H^2(\a)$.  The first such transform maps from $H^2_{\cR}$  to $\cH$ via 
 \begin{equation}\label{19jan4}
f(p,\a)=(\cF_1 F)(p,\a)=\sum_{\un\g\in\G/\G'} F(\un\g p)e^{-2\pi i\a(\g)},\quad F\in H^2_{\cR}.
\end{equation}
\begin{lemma}
The transform \eqref{19jan4} acts unitary from $H^2_{\cR}$ to $\cH$.
\end{lemma}



For the second Fourier transform, we represent each individual $H^2(\a)$ as an exhaustion by invariant subspaces $e^{i\t_{\cR}x}H^2(\a-x\eta)$ in analogue to the classical Paley-Wiener theorem. This generates the Fourier integral
 \begin{equation}\label{19jan5}
f(p,\a)=\int_0^\infty e^{i\t_{\cR}(p)x}e_{\cR}(p,\a-x\eta) \hat f(x,\a) \dd x,
\end{equation}
where $e_{\cR}(p,\a)$ is our special  function (or generalized eigenfunction), i.e., the reproducing kernel related component of the Baker-Akhiezer function
\begin{equation}\label{19jan50}
\Psi(x,t_k;p,\a_0)=e^{i \t_{\cR}(p)x+i \t_{\cR}^{(k)}(p)t_k}{e_{\cR}(p,\a_0-x\eta-t_k\eta^{(k)})}.
\end{equation}
For the precise statement of the main theorem -- including an explicit formula for $e_{\cR}(p,\a)$ -- we need some additional notation and definitions. 

We wish to distinguish the relationship between the functional model for Jacobi operators with that developed in the main theorem below.  In fact, the Fourier transforms above can be considered as a continuous version of the discrete Fourier transform from \cite{SoYud97}, in which the Green function has to be substituted by the Martin function.
Moreover, in this paper we get it  as a limit case, see subsection \ref{ss31}.  We will also substitute the  condition (PW) related to the Green function (see  \eqref{18jan4}) by a similar condition related to the Martin function:
\begin{itemize}
\item[(PW$_M$)] We assume that
\begin{align*}
\mathfrak{w}_M = \sum_{\nabla M(c) = 0} M(c) < \infty.
\end{align*}
\end{itemize}
The condition (PW$_M$) is equivalent to the entropy condition \eqref{eq:entropycondn} and implies the (PW) condition \eqref{18jan4}; see \eqref{eq:qwtoentrpy} and  Corollary \ref{c:qwtowid}.   We will see in subsection \ref{ss43} that it plays an important role in the asymptotics of our special functions $e_{\cR}(p,\a)$.

The functions $e_{\cR}(p,\a)$ are defined by means of  canonical products, see \eqref{19ja1}.
First,  we define the set of divisors
\begin{align*}
	\cD(E):=\{\{\l_j,\e_j\}_{j=1}^\infty: \l_j\in[a_j,b_j],\e_j=\pm 1\}.
\end{align*}
with the identification $(a_j,-1)=(a_j,+1)$ and $(b_j,-1)=(b_j,+1)$, endowed with the product topology of circles. 

Recall that the Blaschke factor in $\bbC_+$ is of the form
$$
b_{z_0}(z)=\frac{z-z_0}{z-\bar z_0}, \quad z_0\in\bbC_+.
$$
The function $\Phi(\bl(z),\l_0)$ is represented by the Blaschke product along the orbit $\{\g(z_0)\}$, $\bl(z_0)=\l_0$, i.e.,
\begin{equation}\label{19jn1}
\Phi(\bl(z),\l_0)=\prod_{\g\in\G}\frac{|b_{\g(z_0)}(i)|}{b_{\g(z_0)}(i)}b_{\g(z_0)}(z).
\end{equation}

To $D\in\cD(E)$ we associate
\begin{equation}\label{19ja1}
e(\l,D)=\sqrt{\prod_{j\ge 1}\frac{(1-\l_j/\l)\Phi(\l,c_j)}{(1- c_j/\l)\Phi(\l,\l_j)}}
\prod_{j\ge 1}\Phi(\l,\l_j)^{\frac{1+\e_j}{2}}.
\end{equation}
Convergence of Blaschke products in \eqref{19ja1} for all $D\in \cD(E)$ corresponds exactly to the Widom condition \eqref{18jan4}, and convergence of the whole product is guaranteed additionally by the finite gap length condition \eqref{21nov0} (for $k=1/2$). Therefore
$e(\l,D)$ is defined as a multi-valued character-automorphic function in $\cS_+$. Its character is denoted by $\a(D)$. 
This map  from divisors $\cD(E)$ to $\G^*$ was introduced in \cite{SoYud94} and called the \textit{generalized Abel map}; we define it precisely below in Definition \ref{d:abelmap}. It is always surjective.  If the generalized Abel map is injective, one can define
\begin{equation}\label{19an1}
e_{\cR}(p,\a(D))=e(\l_{\cR}(p),D).
\end{equation}
Our last assumption on the domain $\cS_+$ (or the set $E$) ensures injectivity of the generalized Abel map:
\begin{itemize}
\item[(DCT)] The DCT property holds in $\cS_+$ \cite[VII.2]{Has83}.
\end{itemize}
This property of Widom domains is not yet completely understood. One of several equivalent conditions to the DCT property holding is:
\begin{itemize}
\item[(DCT)'] $k^\a(p_0,p_0)$ is continuous function in $\a$ for $p_0\in\cR$.
\end{itemize}
An equivalent property is given in Remark \ref{rem30jan1} and  it is  described as the \textit{Direct Cauchy Theorem} (DCT), which holds in the domain, in subsection \ref{ss21}.

\begin{proposition}
 If (DCT) holds, then  the generalized Abel map is a homeomorphism between $\cD(E)$ and $\G^*$. In turn,  \eqref{19an1} defines $e_{\cR}(p,\a)$ on $\G^*$, and it is continuous in this variable for each fixed $p\in \cR$.
\end{proposition}

There is a nice sufficient condition guaranteeing the DCT property holds; specifically, if $E$ is homogeneous, then $\cS_+$ is of Widom type and (DCT) holds. Recall that $E$ is homogenous (in the sense of  Carleson) if there exists $\vk>0$ such that
$$
|E\cap(\xi-\delta,\xi+\delta)|\ge \vk\delta, \quad \forall \xi\in E\ \text{and}\ \forall \d>0.
$$
An example of Widom domain $\cS_+$ such that (DCT) holds but the boundary is not homogeneous is given in \cite{Yud11}.
 
Under our assumptions, we can relate our special functions $e_\cR(p,\a)$ to the reproducing kernels of $H^2(\a)$.  Define $\sqrt{\l}$ in the  domain $\bbC\setminus \bbR_+$ and extend it as character-automorphic function on $\cS_+$. We denote by $\mu_{\cR}(p)$ the corresponding function on $\cR$ with corresponding character by $\fj\in\G^*$,
 $$
 \mu_{\cR}(\un\g p)=e^{2\pi i\fj(\g)}\mu_{\cR}(p).
 $$ 
Since $\mu_{\cR}(p)^2=\l_{\cR}(p)$, $\fj$ is an order 2 element of the group $\G^*$, i.e., $2\fj=0_{\G^*}$.  With this notation, we have the following important identities:
 
 \begin{theorem}\label{t21jan1} Assume that $\cS_+$ is of Widom type with DCT and \eqref{21nov0} holds for $k=1/2$. Then the  reproducing kernel of the space $H^2(\a)$ is of the form
 \begin{equation}\label{20jan3}
k^{\a}(p,p_0)=i\frac{\mu_{\cR}(p)e_{\cR}(p,\a+\fj)\overline{e_{\cR}(p_0,\a)}+e_{\cR}(p,\a)\overline{\mu_{\cR}(p_0)e_{\cR}(p_0,\a+\fj)}}{\l_{\cR}(p)
-\overline{\l_{\cR}(p_0)}};
\end{equation}
alternatively,
 \begin{equation}\label{20jan4}
k^{\a}(p,p_0)=\int_{0}^\infty e^{i(\t_{\cR}(p)-\overline{\t_{\cR}(p_0)}) x}e_{\cR}(p,\a-\eta x)\overline{e_{\cR}(p_0,\a-\eta x)} \dd x.
\end{equation}
 Respectively, the reproducing kernel $K(p,p_0)$ of the space $H^2_{\cR}$ has the forms
  \begin{align}\label{20jan5} 
  K(p,p_0)&=\frac{i}{\mu_{\cR}(p)-\overline{\mu_{\cR}(p_0)}}\int_{\G^*}e_{\cR}(p,\a+\fj)\overline{e_{\cR}(p_0,\a)}\dd\a\\
  &=\frac{i}{\t_{\cR}(p)-\overline{\t_{\cR}(p_0)}}\int_{\G^*}e_{\cR}(p,\a)\overline{e_{\cR}(p_0,\a)}\dd\a. \label{20jan6}
  \end{align}
 \end{theorem}
 
 \begin{remark}
 As a consequence of \eqref{20jan3} and \eqref{20jan4}, $e_{\cR}(p,\a)$ could instead be defined via the reproducing kernels in the following way
  \begin{equation}\label{21jan0}
\frac{e_{\cR}(p,\a)}{e_{\cR}(p_0,\a)}=\lim_{y\to\infty}\frac{k^\a(p,p(iy))}{k^\a(p_0,p(iy))}
\end{equation}
or
 \begin{equation}\label{21jan1}
e_{\cR}(p,\a)\overline{e_{\cR}(p_0,\a)}=\left(\frac{\t_{\cR}(p)-\overline{\t_{\cR}(p_0)}}{i}+\pd_{\eta}\right)k^{\a}(p,p_0)
\end{equation}
While definitions \eqref{21jan0} and \eqref{21jan1} demonstrate the relationship between the special function $e_\cR$ and the reproducing kernels $k^\a$, we prefer our definition \eqref{19an1} because it demonstrates immediately the importance of the DCT condition and is constructive in nature.
 \end{remark}
 
 \subsection{The KdV hierarchy via the functional model}
 
The description of the reproducing kernels of $H^2(\a)$ from Theorem \ref{t21jan1} yields, in fact, 
the representation for the Weyl-Titchmarsh  $m$-function:
 
 \begin{corollary} In the setting of Theorem \ref{t21jan1},
 \begin{equation}\label{20jan6}
m^\a_{+}(\l_{\cR}(p))=m^\a_+(0)+i\mu_{\cR}(p)\frac{e_{\cR}(p,\a+\fj)}{e_{\cR}(p,\a)}
\end{equation}
together with the asymptotics
$$
\lim_{\mu\to\infty}\left(m_+^\a(\mu^2)-i\mu\right)=0,
$$
defines a single-valued function $m_+^\a(\l)$ in $\cS_+$ with positive imaginary part in $\bbC_+$.  This function possesses the reflectionless property \eqref{30jan10}.
Moreover,
 if the $k$-th gap length condition \eqref{21nov0} holds, then we may define the system of functions $\{\chi_n(\a)\}_{n=0}^{2k}$ by the following asymptotic expansion:
 \begin{equation}\label{20jan7}
\lim_{\mu\to\infty}\mu^{n+1}\left\{\frac{m_+^\a(\mu^2)-m_+^\a(0)}{i\mu}-1-\frac{\chi_0(\a)}{\mu}-\dots\frac{\chi_{n-1}(\a)}{\mu^n}\right\}=\chi_{n}(\a).\quad
n=0,\dots,2k,
\end{equation}
 In particular, $\chi_0(\a)=i m_+^\a(0)$.

\end{corollary}

Following Dubrovin et al. \cite{DKN}, we use the system of functions $\{\chi_n(\alpha)\}_{n=0}^{2k}$ to generate the $k$-th element of the KdV hierarchy $\KdV_k(\alpha)$.  Specifically, we show that
the generalized eigenfunction $e_{\cR}(p,\a)$ itself satisfies the following differential relation
\begin{equation}\label{pe17oct}
(\t_{\cR}^{(k)}(p)+i\pd_{\eta^{(k)}})e_{\cR}(p,\a)
=A_k(\l_{\cR}(p),\a) \mu_{\cR}(p)e_{\cR}(p,\a+\fj)-B_k(\l_{\cR}(p),\a)
 e_{\cR}(p,\a),
\end{equation}
where  $A_k(\l,\a)$ and $B_k(\l,\a)$ are polynomials of degree $k$,
$$
A_k(\l,\a)=\sum_{n=0}^k \cA_{k}(\a)\l^{k-n}, \quad B_k(\l,\a)=\sum_{n=0}^k \cB_{k}(\a)\l^{k-n}.
$$
 We compute their coefficients by asymptotics  \eqref{20jan7} and consequently obtain
\begin{align}\label{20jan8}
\begin{bmatrix} 1\\ \cA_1(\a)\\ \vdots\\ \cA_{k}(\a)
\end{bmatrix}&=\begin{bmatrix}
1 & & & \\
\chi_1(\a)&1& & \\
 \ddots& \ddots& \ddots&  \\
 \chi_{2k-1}(\a)&\ddots&\chi_1(\a)&1\\
\end{bmatrix}^{-1}\begin{bmatrix} 1\\ 0\\ \vdots\\ 0
\end{bmatrix},
\\
\label{20jan9}
\begin{bmatrix} \cB_0(\a)\\ \cB_1(\a)\\ \vdots\\ \cB_{k}(\a)
\end{bmatrix}
&=\begin{bmatrix}
\chi_0(\a) & & & \\
\chi_2(\a)&\chi_0(\a)& & \\
 \ddots& \ddots& \ddots&  \\
 \chi_{2k}(\a)&\ddots&\chi_2(\a)&\chi_0(\a)\\
\end{bmatrix}\begin{bmatrix} 1\\ \cA_1(\a)\\ \vdots\\ \cA_{k}(\a)
\end{bmatrix}.
\end{align}
Moreover, the partial derivatives in the $\eta$-direction of the coefficients exist; specifically, 
\begin{equation}\label{20jan10}
\cB_n(\a)=\left(\frac {i\pd_{\eta}} 2+\chi_0(\a)\right)
 \cA_n(\a), \quad n=0,\dots, k.
\end{equation}
Equation \eqref{pe17oct}, together with \eqref{20jan8}--\eqref{20jan10}, essentially leads to our main result.

\begin{theorem}[Main Theorem]\label{t:mainthm}
Let   (PW$_M$) and (DCT) conditions hold in a regular domain $\cS_+$. If the $k$-th gap length condition \eqref{21nov0} holds,
then for $\a\in\G^*$, the following differential operators $L_\a$ and $P^{(k)}_\a$ are well defined by the decomposition
\eqref{20jan7} 
$$
 L_\a=-\pd_x^2+V^\a(x), \quad P^{(k)}_{\a}:=i\sum_{n=0}^k L_\a^{k-n}\left(A^\a_n(x)\pd_x+\frac 3 2 (A_n^\a)'(x)\right),
$$
where
$
V^\a(x)=2\cA_1(\a-\eta x),\quad 
A_n^\a(x)=\cA_n(\a-\eta x)
$,  see \eqref{20jan8}.

Let
$$
f(p,\a)=\int_{0}^\infty e_{\cR}(p,\a-\eta x)e^{i\t_{\cR}(p)x}\hat f(x,\a) \dd x.
$$
Then in the Fourier transforms  \eqref{19jan4} and \eqref{19jan5} the multiplication operators by $\l_{\cR}$ and $\t_{\cR}^{(k)}$ are of the form
\begin{align}
\l_{\cR}(p)f(p,\a)=\int_{0}^\infty e_{\cR}(p,\a-\eta x)e^{i\t_{\cR}(p)x}
L_\a \hat f(x,\a) \dd x,\\
(\t_{\cR}^{(k)}(p)+i\pd_{\eta^{(k)}})f(p,\a)=\int_{0}^\infty e_{\cR}(p,\a-\eta x)e^{i\t_{\cR}(p)x}
\{i\pd_{\eta^{(k)}}+P^{(k)}_\a\}\hat f(x,\a) \dd x.
\end{align}
If the $(k+1)$-th gap length condition \eqref{21nov0} holds, the commutant relation between these operators corresponds to the Lax-pair representation for the k-th KdV hierarchy equation
$$
i\pd_{\eta^{(k)}} L_\a=[L_\a,P^{(k)}_\a].
$$
Respectively, the time evolution is given by $L_{\a(t)}=L_{\a-\eta^{(k)}t}$.
\end{theorem}

\begin{remark}The $P_\a^{(k)}$ defined in the main theorem are indeed the Lax pair operators corresponding to the KdV hierarchy; in fact, this is relatively straightforward to recover from the definition of the coefficients $\chi_n$ (cf. Dubrovin et al, \cite[Section 30.2]{DFN}).  In the general setting, for a sufficiently smooth potential $V$, one can define coefficients $\chi_n(V)$ by formally expanding the $m$-function of $L_V$.  These coefficients are polynomial in $V$ and its derivatives by the Ricatti equation.  Then if one defines
\begin{align*}
\KdV_k(V) = \frac{\dd}{\dd x} \frac{\d}{\d V} \int \chi_{2k+3} \dd x
\end{align*}
(where $\d/\d V$ is a variational derivative with respect to $V$), one can prove the existence of operators $P_V^{(k)}$ such that
\begin{align*}
-i[L_V, P_V^{(k)}] = \KdV_k(V).
\end{align*}
In our case, we have from the expansion \eqref{20jan7} that $P_\a^{(k)} = P_{V^\a}^{(k)}$.  For example, one can verify directly from our results that
\begin{align*}
P_\a^{(1)} = i\left(-\partial_x^3 + \frac{3}{2}V^\a \partial_x + \frac{3}{4}(\partial_x V^\a)\right)
\end{align*}
(cf. Theorem \ref{t:chi0diff}), which yields the traditional KdV equation in the Lax pair formulation.
\end{remark}

\begin{remark}\label{rem30jan1}
A natural question: is there a counterpart for the second sheet $\cS_-$, $\cS=\cS_+\cup\cS_-\cup\pd \cS_+$, in the infinite dimensional case? Again, using universal covering we can reduce the answer to a well known object in the theory of Hardy spaces in $\bbC_+$ \cite[Lecture II, Sect. 1]{Nik}. We say that a function $f$ of bounded characteristic in the upper half plane has a \textit{pseudocontinuation} in the lower half plane $\bbC_-$ if there is a function $g$ of bounded characteristic in the lower half plane such that
$$
f(x+i0)=g(x-i0) \quad \text{for almost all} \ x\in\bbR \ \text{w.r.t.}\ dx.
$$
An equivalent statement: there exists a function $h$  of bounded characteristic  in $\bbC_+$ such that
$$
\overline{f(x+i0)}=h(x+i0) \quad \text{for almost all} \ x\in\bbR \ \text{w.r.t.}\ dx.
$$
In this case $\overline{g(\bar z)}=h(z)$, $z\in\bbC_+$. Using this notion we can say that $e_{\cR}(p,\a)$ possesses a pseudocontinuation in the sense that $\overline{e_{\cR}(p,\a)}$, $p\in\pd\cR_+$, can be extended in $\cS_+$ as a function of bounded characteristic. We can write this extension explicitly. We introduce the \textit{Widom function}, which is the Blaschke product
\begin{equation}\label{22jan1}
\cW_{\cR}(p)=\prod_{j\ge 1}\Phi(\l_{\cR}(p),c_j)
\end{equation}
and denote its character by $\a_{\cW}$.
In this case
\begin{equation}\label{22jan2}
\overline{e_{\cR}(p,\a)}=\frac{e_{\cR}(p,\a_{\cW}-\a)}{\cW_{\cR}(p)}
\end{equation}
for almost all $p\in\pd\cR$ w.r.t. $d\t_{\cR}$. The last relation is easy to explain in the following way: $\cW_{\cR}(p)\overline{e_{\cR}(p,\a)}$, 
$p\in\pd\cR$,
has a  form of the canonical product with a certain $D_*\in\cD(E)$. Therefore this is $e_{\cR}(p,\a(D_*))$. It remains to note that by the definition the character of this function is $\a_{\cW}-\a$, that is, $\a(D_*)=\a_{\cW}-\a$. 

The relation \eqref{22jan2} is closely related with a description of the orthogonal complement of the Hardy spaces. Let us define
$$
H_{\bot}^2(\a):=L^2_{d\t_{\cR}}(\a)\ominus H^2(\a).
$$
In this case for an arbitrary Widom surface 
$$
\cW_\cR(p) \overline{g(p)}\in H^2(\a_{\cW}-\a)\ \text{for all}\ g\in H_{\bot}^2(\a).
$$
But
$$
H_{\bot}^2(\a)=\{g: \cW_\cR(p) \overline{g(p)}\in H^2(\a_{\cW}-\a)\}
$$
\textit{if and only if DCT holds}.

A notion of pseudocontinuation is very closely related with the notion of the \textit{reflectionless property} in the theory of ergodic operators. The role of this property in the spectral theory was completely understood in \cite{Rem11}, see also \cite{PoRem09}. Equation \eqref{22jan2}  implies that
\begin{equation}\label{30jan10}
\overline{m_+^\a(\l+i0)}=-m_-^\a(\l+i0)\ \text{for almost all}\ \l\in E,
\end{equation}
where $m_-^\a(\l):=m_+^{\a_{\cW-\a}}(\l)$. This is exactly means that the Nevanlinna class functions $m_{\pm}^\a(\l)$ possess  reflectionless property on $E$, see \eqref{15jan1}.
\end{remark}

With this discussion in hand, Theorem \ref{t:introthm} follows almost immediately from our main theorem:
\begin{proof}[Proof of Theorem \ref{t:introthm}]
If $V$ is ergodic such that $L_V$ has absolutely continuous spectrum $E$, then $L_V$ is reflectionless on $E$ by Kotani theory.  Since $E$ is homogeneous, $\cS_+$ is a regular domain of Widom type with DCT; consequently, $V = V^\a$ for some $\a \in \Gamma^*$.  Since $E$ satisfies the $(k+1)$-th moment condition, $V(\cdot,t_k) := V^{\a + \eta^{(k)}t_k}$ satisfies \eqref{eq:kdvh.1}, \eqref{eq:kdvh.2} and is almost-periodic in $x$ and $t_k$.
\end{proof}

\section{Preliminaries and elements of spectral theory}
Let $V:\bbR\to\bbR$ be a continuous real-valued function which is bounded from below.  Consider the associated one-dimensional Schr\"odinger operator 
\begin{align*}
	L_V=-\pd_x^2+V(x),\quad x\in\bbR
\end{align*}
which is an unbounded self-adjoint operator on $L^2(\bbR)$ with domain 
\begin{align*}
	\dom(L_V)=\{f\in L^2(\bbR): f,f'\in\text{AC}_{\text{loc}}(\bbR), L_Vf\in L^2(\bbR)\}
\end{align*}
and resolvent (or Green's) function
\begin{align*}
	R(\l;x_0,x_1)=\langle(L-\l)^{-1}\d_{x_1},\d_{x_0} \rangle,\quad \l\in\bbC\setminus\sigma(L_V),\ x_0,x_1\in\bbR.
\end{align*}
By  $u_{1,2}(x,x_0,\l)$ we denote the unique fundamental system satisfying 
\begin{align*}
	L_Vu_{1,2}(x,x_0,\l)=\l u_{1,2}(x,x_0,\l),\quad \l\in\bbC,
\end{align*} 
subject to the boundary conditions at $x_0\in\bbR$
\begin{align*}
\begin{bmatrix}
	u_1(x_0,x_0,\l)& u_2(x_0,x_0,\l)\\
	\pd_x u_1(x_0,x_0,\l)&\pd_x u_2(x_0,x_0,\l)
	\end{bmatrix}=
	\begin{bmatrix}
	1&0\\
	0&1
	\end{bmatrix}.
\end{align*}

The condition that $V$ is bounded from below implies that $L_V$ is in the limit point case at $+\infty$ and $-\infty$, and hence we may define the Weyl-Titchmarsh functions $m_{\pm}(\l;x_0)$ uniquely by the condition
$$
u_\pm(x,x_0,\l)=u_1(x,x_0,\l)\pm m_\pm(\l;x_0)u_2(x,x_0,\l)\in L^2([x_0,\pm\infty).
$$
The diagonal of the resolvent function is given in terms of $m_\pm$ by 
\begin{equation}\label{18may1}
-\frac 1{R(\l;x,x)}=m_+(\l;x)+m_-(\l;x).
\end{equation}
If $x_0=0$, we use the abbreviations $R(\l)=R(\l;0,0)$ and $m_\pm(\l)=m_\pm(\l;0)$. 

Since $V$ is bounded from below, so is the spectrum $\sigma(L_V)$ of $L_V$,  i.e. $\inf\sigma(L_V) > -\infty$.  By translation, we may assume without loss of generality that $\inf\sigma(L_V) = 0$.
We choose the branch of $\sqrt{\l}$ such that $i\sqrt{\l}<0$ if $\l<0$. The functions $R$ and $m_\pm$ have asymptotic expansions of the form 
\begin{align}\label{eq:asympGreen}
R(\l,x,x)=\frac i{2\sqrt{\l}}+ \frac{i V(x)}{4\sqrt{\l}^3}+o(\sqrt{\l}^{-3})
\end{align}
and 
\begin{align}\label{eq:asympmplus}
	m_{\pm}(\l,x)=i\sqrt{\l}- \frac{i V(x)}{2\sqrt{\l}}+o(\sqrt{\l}^{-1}).
\end{align}
as $\l\to-\infty$. 

\subsection{Some inverse spectral theory}\label{ss21}


We recall from the introduction our preliminary assumptions on $E$: namely, $E \subset \bbR_+$ is a closed set of positive Lebesgue measure of the form
\begin{align}\label{def:setE}
E=\bbR_+\setminus\bigcup_{j=1}^\infty(a_j,b_j)
\end{align} 
without isolated points, such that the domain $\cS_+ = \bbC \setminus E$ satisfies (R), ($1/2$-GLC), (PW), and (DCT).  We now provide an alternative characterization of the DCT property.

Recall that a meromorphic function $f$ in $\bbC_+$  is said to be of bounded characteristic if it can be represented as the ratio of two bounded analytic functions, $f=f_1/f_2$. The function is of Smirnov class if in addition $f_2$ is an outer function, cf. e.g. \cite{Gar07}. We say that a function $F$ on $\cS_+$ belongs to the Smirnov class $\cN_+=\cN_+(\cS_+)$, if $f=F\circ \bl$ is of Smirnov class in $\bbC_+$.  
	 We say that the Direct Cauchy Theorem  holds in $\cS_+$ if for every $F\in \cN_+$ with
	 \begin{align*}
	 	\oint_E\left|\frac{F(\x)}{1+\x}\right||\dd \x|<\infty,
	 \end{align*}
	 	we have
	 \begin{align*}
	 	\frac{1}{2\pi i}\oint_E\frac{F(\x)}{\x-\l_0}\dd \x=F(\l_0), \quad \l_0 \in \cS_+.
	 \end{align*}
	 Here the contour integral is shorthand for integrating over both the ``top" and ``bottom" of $E$, i.e.
	 \begin{align*}
	 \oint_E F(\x) \dd \x := \int_E F(\x+i0)\dd \x + \int_E F(\x - i0) \dd \x.
	 \end{align*}

Later, we will add to these conditions ($k$-GLC) for higher $k$, but for now these conditions will suffice.  Under these assumptions, a certain class of Schr\"odinger operator is particularly amenable to inverse scattering techniques:
\begin{definition}
	We call $L_V$ reflectionless on $A\subset \bbR$, with $|A|>0$  if 
	\begin{equation}\label{15jan1}
		m_+(\x+i0)=-m_-(\x-i0),\quad \text{for a.e.}\;\x\in A.
	\end{equation}
	For a given set $E$, we define the set of potentials
	$$
	\cV(E):=\{V: \sigma(L_V)=E \text{ and } L_V \text{ is reflectionless on E}\}.
	$$
\end{definition}

Note that \eqref{18may1} and \eqref{15jan1} imply
\begin{equation}\label{15jan2}
\Re R(\x+i0):=\Re R(\x+i0,0,0)=0\quad \text{for a.e.}\;\x\in E.
\end{equation}
Since $R(\l)$ is real and monotonic in the gaps $(a_j, b_j)$, for each $j \in \bbN$ there exists a unique ``Dirichlet eigenvalue"
$\l_j\in[a_j,b_j]$ such that
\begin{align*}
\frac 1\pi \arg R(\x)=\begin{cases}
\frac 1 2,\quad &\x\in E,\\
1,\quad &\x\in(a_j,\l_j),\\
0,\quad &\x\in(\l_j,b_j).
\end{cases}
\end{align*} 
Therefore the resolvent function  can be represented by 
\begin{align}\label{eq:intRepGreen}
	R(\l)= C\frac{i}{\sqrt{\l}}e^{-\int_{0}^{^\infty}\left\{\frac{1}{\xi-\l}-\frac{\xi}{1+\xi^2}\right\}\ff(\xi)\dd \xi},\quad \ff(\xi):=\frac 1 2 - \frac 1\pi \arg R(\xi), \
	 C>0.
\end{align}
Moreover, by the finite length gap condition and the normalization  \eqref{eq:asympGreen} we can compute $C$ and obtain the product formula
\begin{align}\label{eq:multRep}
R(\l)=\frac{i}{2\sqrt{\l}}\prod_{j=1}^{\infty}\frac{1-\l_j/\l}{\sqrt{(1-a_j/\l)(1-b_j/\l))}}.
\end{align}
In particular, \eqref{eq:asympGreen} and \eqref{eq:intRepGreen} yield the trace formula
\begin{align*}
	V(0)=\sum_{j=1}^\infty a_j+b_j-2\l_j.
\end{align*}
We now state a lemma which is fundamental in what follows; specifically, under the conditions above the infemum of the spectrum is a regular point for the Weyl $m$-functions:
\begin{lemma}\label{lem:limitMfunction}
	Let $V\in\cV(E)$ and $m_\pm$ be the corresponding Weyl-Titchmarsh functions. Then the following limits exist and are finite:
	\begin{align*}
		\lim\limits_{\x\to 0^-}m_\pm(\x)=m_\pm(0).
	\end{align*}
	Moreover
	\begin{align*}
		m_+(0)=-m_-(0).
	\end{align*}
\end{lemma}
\begin{proof}
	First, we note that $\lim_{\lambda\to 0}R(\lambda)=\infty$. Indeed, if for
	$$
	R(\lambda)=\frac i{2\sqrt{{\lambda}}}\prod_{j\ge 1}\frac{1-\l_j/\l}
	{\sqrt{(1-a_j/\l)(1-b_j/\l)}}
	$$
	we have $\lim_{\lambda\to 0}R(\lambda)<\infty$, then
	$$
	\lim_{\l\to 0}\frac{\l}{\l R(\l)} >0.
	$$
	By the product representation \eqref{eq:multRep}, the function 	$w(\l)=-\frac 1{\l R(\l)}$ has the interlacing property and is of Nevanlinna class.  Hence, the Nevanlinna measure $\dd\sigma$ corresponding to the $w$ has a mass point at the origin. But, since $R$ is reflectionless (see \eqref{15jan2}),
	$$
	\Re w(\x+i0)=0,\quad\text{for a.e. } \l\in E.
	$$ 
	On the other hand, if (DCT) holds, the singular part of $\dd\sigma$ can not be supported on $E$ \cite[Theorem 1]{PoRem09}. Consequently, it must be the case that $\lim_{\lambda\to 0}R(\lambda)=\infty$. 
	Since $m_+$ and $m_-$ are Nevanlinna class functions with Nevanlinna measures supported on $\bbR_+$, they are increasing functions on $\bbR_-$. Relation \eqref{18may1} then concludes the proof. 
\end{proof}
The following theorem was shown by Sodin and Yuditskii \cite{SoYud95}:
\begin{theorem}
The spaces $\cV(E)$ and $\cD(E)$ are homeomorphic. 
\end{theorem}

A similar result was obtained by Sodin and Yuditskii for Jacobi matrices \cite{SoYud97} by associating to each reflectionless Jacobi matrix $J$ a Hardy space of character automorphic functions. This description was carried over to continuous Schr\"odinger operators by Damanik and Yuditskii \cite{DY16}.  We will present and use their result in the following in a way that is convenient for our purpose. 

First, we can explicitly describe the homeomorphism $\cV(E) \to \cD(E)$ mentioned above:
\begin{lemma}\label{lem:Abel1}
	Let $V\in\cV(E)$, and denote by $R(\l)$ and $m_+$ the corresponding resolvent and Weyl-Titchmarsh functions, respectively. Then there exists a divisor $D\in\cD(E)$ such that 
	\begin{align}\label{eq:intmplus}
		m_+(\l)=-\frac{1}{2R(\l)}+\sum_{\{ j: \l_j\in(a_j,b_j)\}}\frac{\l\s_j\e_j}{2(\l_j-\l)}+m_+(0),
	\end{align}
	where 
	\begin{align}\label{eq:intmplus2}
		\s_k=2\frac{\sqrt{(\l_k-{a_k})(b_k-\l_k)}}{\sqrt{\l_k}}\prod_{j\not=k}\frac{\sqrt{(1-a_j/\l_k)(1-b_j/\l_k)}}
		{1-\l_j/\l_k}
		>0.
	\end{align}
	Moreover, 
	\begin{align}\label{eq:intmplus3}
		\sum_{j=1}^\infty{\s_j}<\infty \quad\text{and}\quad m_+(0)=\frac 1 2\sum_{j=1}^\infty\sigma_j\e_j.
	\end{align}
\end{lemma}
\begin{proof}
	By the reflectionless property, the absolutely continuous part of the measure corresponding to $-R(\l)^{-1}$ should be equally divided between $m_+$ and $m_-$. The $\l_j$'s are given by \eqref{eq:multRep}.  
	Set 
	$$
		R_{11}(\l):=\frac{(m_+(\l)-m_+(0))(m_-(\l)-m_-(0))}{m_+(\l)+m_-(\l)}.
	$$
	Again by the reflectionless condition, we encounter as in \eqref{eq:intRepGreen} that there exist $ \l^{(1)}_j\in[a_j,b_j]$ such that 
	\begin{align}\label{eq:R11}
			R_{11}(\l)=\frac{i\sqrt{\l}}{2}\prod_{j=1}^{\infty}\frac{\l- \l^{(1)}_j}{\sqrt{(\l-a_j)(\l-b_j)}}.
	\end{align}
	Therefore, provided $\l_j \in (a_j,b_j)$, it is not possible that $m_-$ and $m_+$ have a pole at $\l_j$ simultaneously.  The $\e_j$ are chosen in order to add or cancel the pole at $\l_j$. By the previous lemma it is possible to use $0$ as a normalization point, which concludes the proof of \eqref{eq:intmplus} and \eqref{eq:intmplus2}. Finally, \eqref{eq:intmplus3} follows by \eqref{eq:asympGreen} and \eqref{eq:asympmplus}.
\end{proof}

Both of these spaces are, in turn, homeomorphic to the Fuchsian dual $\Gamma^*$ by way of the generalized Abel map, which we now describe.  

We defined the functions $\Phi(\lambda, \lambda_0)$ in \eqref{18jan10}.  This function is related to the potential theoretic Green's function of $\cS_+$ by \eqref{18jan10}.
Let $E_k=[b_k,\infty)\cap E$ and let $\g_k\in\G$ correspond to the contour starting at $-1$ and containing the set $E_k$. Moreover, let $\omega(\l,I)$ denote the harmonic measure on $\cS_+$ of the set $I$ evaluated at the point $\l$. By \eqref{18jan10} we see that the character $\nu_{\l_0}$ of $\Phi(\lambda, \lambda_0)$ can be given by means of the harmonic measure, i.e.,
$$
\Phi(\g_k(\l), \l_0) = e^{2\pi i\omega(\l_0,E_k)}\Phi(\l, \l_0).
$$

The following generalized Abel map was introduced in \cite{SoYud94}.
\begin{definition}\label{d:abelmap}
	The Abel map $\cA:\cD(E)\to \G^*$ is defined by
	\begin{align*}
	\cA(D,\g_k)=\frac 1 2\sum_{j\ge 1}\omega(\l_j, E_k)\e_j \mod 1
	\end{align*}
	This map is a homeomorphism \cite{SoYud95}.
\end{definition}

It will be convenient to define the Abel map as being shifted by a fixed character corresponding to the divisor $D_c=\{(c_j,-1)\}_{j\geq 1}$. That is, we set
$$
\a(D)=\cA(D)-\cA(D_c).
$$
We can thus put in correspondence potentials $V \in \cV(E)$ and characters $\alpha \in \Gamma^*$; we will later describe explicitly such a map (and various other important spectral quantities).

\subsection{Generalized eigenfunctions and a Wronskian identity}

The symmetric Martin's function $M(\l)$ related to infinity is a positive, symmetric (i.e., $M(\l)=M(\overline{\l})$), harmonic function in $\cS_+$ which vanishes on $E$. It is unique under the normalization $M(\l)=\sqrt{\l}+o(1/\sqrt{\l})$ as $\l\to -\infty$. $M$ has precisely one critical point $c_j$ in each gap $(a_j,b_j)$ of $E$. 

$M(\l)$ can also be represented in terms of a conformal mapping. Fix $h_j,\eta_j>0$ with $\eta_j<\eta_{j+1}$. To this data we associate the comb 
\begin{align}\label{def:MartinComb}
	\Pi_+=\{\vt: \ \Re\vt>0,\ \Im \vt >0\}\setminus\cup_{j\geq 1}\{\vt: \ \vt=\pi\eta_j+ih, h \in (0, h_j]\}.
\end{align}
By the Riemann mapping theorem, there exists a conformal map $\Theta:\bbC_+\to \Pi_+$, which is unique under the normalization that $\Theta(0)=0$ and $\Theta(\l)=\sqrt{\l}+o(1/\sqrt{\l})$ as $\l\to -\infty.$ Noting that $\Theta$ can be continuously extended to the boundary, we set $E=\Theta^{-1}(\bbR_+)$. Moreover, $\Theta$ can be extended to $\cS_+$ as an additive automorphic function
\begin{equation}\label{5amay3}
\Theta(\g_j(\l))=\Theta(\l)+2\pi\eta_j.
\end{equation}
With this definition, one has $\Im\Theta(\l)=M(\l)$. The tops of the needles $\eta_j+ih_j$ correspond to the critical values $M(c_j)$.  Conversely, for every given set $E$ there exists a corresponding conformal map $\Theta$ with the above properties for some comb $\Pi_+$, cf. e.g. \cite{Yud11}.  In fact, for $E$ as in \eqref{def:setE}, $\Theta'(\l)$ can be given explicitly. Specifically, 
\begin{align}\label{eq:thetaPrime}
\Theta'(\l)=\frac{-1}{2\sqrt{\l}}\prod_{j=1}^{\infty}\frac{1-c_j/\l}{\sqrt{(1-a_j/\l)(1-b_j/\l)}}.
\end{align}

We likewise define the Widom function 
$$
\cW(\l)=\prod_{j\ge 1}\Phi(\l, c_j);
$$
such that $\cW_\cR(p) = \cW(\l_{\cR}(p))$; it is the inner part of $\Theta '$.  With this notation, we can now define our generalized functions $e_\a$ on $\cS_+$ and examine some of their properties:
\begin{theorem} Let $\fj$ be the character generated by $\sqrt{\l}$ in $\cS_+$ and $\a=\a(D)$. We define the character automorphic functions 
	$$
	e_\a(\l)= e(\l, D) = \sqrt{\prod_{j\geq 1}\frac{(1-\l_j/\l)\Phi(\l,c_j)}{(1-c_j/\l)\Phi(\l,\l_j)}}\prod_{j\geq1}\Phi(\l,\l_j)^{\frac{1+\e_j} 2}.
	$$
	and
	\begin{align}\label{eq:dualBasis}
			\tilde e_\a(\l)=e_{\a(\tau D)}(\l),\quad \tau(D)=\{(\l_j,-\e_j)\}_{j=1}^\infty.
	\end{align}
	Then 
	\begin{equation}\label{15may3}
	\cW(\x)\overline{e_\a(\x)}=\tilde e_\a(\x)\quad \text{a.e. } \x \in E.
	\end{equation}	
	Moreover, for every $D$ there exists $\a\in\G^*$ such that we have
	\begin{equation}\label{15may7}
	m_+(\l,D)-m_+(0,D)=i\sqrt{\l} \frac{e_{\a+\fj}(\l)}{e_{\a}(\l)},
	\end{equation}
	and the following Wronskian identity holds for all $\a\in\G^*$:
	\begin{align}\label{eq:WronskId}
		e_{\a+\fj}(\l)\tilde e_\a(\l)+e_\a(\l)\tilde e_{\a+\fj}(\l)=\frac{\cW(\l)}{\sqrt{\l}\Theta'(\l)}.
	\end{align}
\end{theorem}
\begin{proof}
Note that by the definition of the Abel map, $e_\alpha$ indeed has character $\alpha$; compare, e.g., \cite[Section 10.2]{SoYud97}.

	Let $m_\pm(\l)=m_\pm(\l,D)$ and $\a(\g)=\cA(D,\g)-\cA(D_c,\g)$.
	By \eqref{eq:R11}, for $\x \in E$ we have
	$$
	\frac{ R_{11}(\x)}{R_{00}(\x)}=|m_+(\x)-m_+(0)|^2=\x\prod_{j=1}^{\infty}\Big(\frac{\x-\tilde{\l}_j}{\x-\l_j}\Big)^2.
	$$
	By \cite[Theorem 4.1]{KotaniLast} $m_+(\bl(z))$ is a function of bounded characteristic such that its inner part represents a ratio of Blaschke products. Therefore
	$$
	m_+(\l)-m_+(0)=C(\a)i\sqrt{\l}\frac{e_{\a+\fj}(\l)}{e_\a(\l)}.
	$$
	Then
	\begin{equation}\label{15may10}
	-m_-(\x)-m_+(0)=\overline{m_+(\x)-m_+(0)}=-C(\a)i\sqrt{\x}\frac{\tilde e_{\a+\fj}(\x)}{\tilde e_\a(\x)}, \quad \text{a.e. } \x \in E.
	\end{equation}
	Since
	$$
	\lim_{\l \to -\infty}e_\a(\l)\tilde e_\a(\l)=1,
	$$
	by \eqref{eq:asympmplus} we have
	$$
	C(\a)^2=\lim_{\l\to-\infty} \frac{m_+(\l) m_-(\l)}{\l}=1
	$$
	and \eqref{15may7} is proved. It remains to prove \eqref{eq:WronskId}. We have
	$$
	R^\a(\l)=i\frac{e_\a(\l)\tilde e_\a(\l)\Theta'(\l)}{\cW(\l)}.
	$$
	Using \eqref{15may10} we obtain that 
	\begin{align*}
		R^\a(\l)=\frac{-1}{m^\a_+(\l)-m_+^\a(0)+m^\a_-(\l)+m_+^\a(0)}=\frac{i{e_\a(\l)\tilde e_\a(\l)}}{\sqrt{\l}
			({e_{\a+\fj}(\l)\tilde e_\a(\l)}+{e_\a(\l)\tilde e_{\a+\fj}(\l)})},
	\end{align*}
	which concludes the proof.
\end{proof}

\section{Functional models and the space $H^2(\a)$}

Let $\G'$ denote the commutator subgroup of $\G$. It is generated by the commutators
$$
[\g_1,\g_2]=\g_1\g_2\g_1^{-1}\g_2^{-1},\quad \g_1,\g_2\in\G.
$$
and can be given by
\begin{align}
	\G'=\bigcap_{\a\in\G^*}\ker \a.
\end{align}
Elements of $\G/\G'$ are denoted by $\un\g$. We have
\begin{equation}\label{5may2}
\un\g=\{\g_1^{n_1}\g_2^{n_2}\dots\g_m^{n_m}\oc\g:\ \oc\g\in\G'\}\simeq 
(n_1,n_2,\dots, n_m,0,,0,\dots)\in (\bbZ^\infty)_0,
\end{equation}
that is, $\G/\G'\simeq (\bbZ^\infty)_0$. 

The subgroup $\G'$ of $\G$ is the smallest normal subgroup such that the quotient $\G/\G'$ is abelian. By the duality theorem \cite[Theorem 24.2]{HeRo79} $(\G^*)^*=\G/\G'$ and we have the Fourier transforms
\begin{equation}\label{5may3}
f(\a)=\sum\hat f_{\un\g}e^{2\pi\a(\g)},\quad f_{\un\g}=\int_{\G^*}f(\a)e^{-2\pi i\a(\g)}d\a,
\end{equation}
where $\{f_{\un\g}\}\in\ell^2_{\G/\G'}$ and $f(\a)$ belongs to the space $L^2_{d\a}$ of square-integrable functions w.r.t. the Haar measure $d\a$ on $\G^*$; see e.g. \cite{Ru90}.

Consider the Riemann surface $\cR=\bbC_+/\G'$. $\cR$ is called the (universal) Abelian covering for the Riemann surface $\bbC_+/\G$.  Points on $\cR$ are denoted by $p$ and the projection from $\cR$ onto $\bbC_+/\G$ is denoted by $\pi$.  For $z\in\bbC_+$ we associate the point $p=p(z)$ corresponding to the orbit $\{\oc\g(z)\}_{\oc\g\in\G'}$. The group $\G/\G'$ acts on $\cR$ in the natural way: $\un\g p=\{\oc\g(\g(z))\}_{\oc\g\in\G'}$.  

\begin{remark}
We can describe $\cR$ in terms of the covering maps $\l_\cR$ and $\t_\cR$. First we fix a fundamental domain $\cF$ for the covering $\bl$. For a given set $E$ there exists a system of non-intersecting half discs $\bbD^+_j=\{z\in\bbC_+: |z-\zeta_j|<r_j\}$, $\zeta_j\in\bbR_+, r_j>0$ such that $\bbC_+$ can be mapped conformally onto $\cF_+=\{z: \Re z>0, \Im z>0\}\setminus\cup_{j\geq 1}\bbD^+_j$, with the following properties:
\begin{align*}
	\phi(0)=0,\quad \phi(\infty)=i\infty,\quad \phi((a_j,b_j))=\partial\bbD_j\cap\bbC_+.
\end{align*}
By the symmetry principle we extend $\phi$ as a conformal mapping from 
$$
\cF=\cF_+\cup\cF_-\cup i\bbR_+,\quad\text{where }\cF_-=\{-\overline{z}:\quad z\in\cF_+\}
$$
to $\bbC\setminus \bbR_+$. By extending $\phi$ with respect to the gaps $(a_j,b_j)$ we obtain $\bl$, respectively we can describe the action of $\G$. We fix $\cF$ as a fundamental domain of $\bl$. Let $\tilde\G$ denote a system of representatives of $\G/\G'$. Then 
$$
\oc\cF=\cup_{\tilde \g\in\tilde \G}\tilde\g(\cF)
$$
is a fundamental domain for $\G'$. To describe $\cR$ by means of the function $\l_\cR$ we take $(\bbZ^\infty)_0$ copies $\cS_{\un\g}$ of  $\bbC\setminus \bbR_+$ cut along the gaps $(a_j,b_j)$. We fix the zero sheet  $\cS_{\un\iota}$ corresponding to $\cF$, where $\l_\cR$ is one-to-one. Let $z\in\cF$ and $\g_j\in\G$ corresponding to the closed loop $\tilde\g_j$ that passes through $(a_j,b_j)$.  Passing from $p(z)$ to $p(\g_j(z))$ means that we pass from $\un\iota$ to $\un\g_j$. Generally, two sheets corresponding to $\un\g$ and  $\un{\g\gamma_j^{\pm 1}}$ are glued together at the gap $(a_j,b_j)$.

Similarly, we can describe $\cR$ by means of $\t_\cR$. 
The sheet $\un\iota$ is given by
$$
\Pi_{\un\iota}=\Pi_+\cup \Pi_-\cup i\bbR_+ \quad\text{where}\ \Pi_-=\{-\overline{\vt}:\ \vt\in\Pi_+\},
$$
see \eqref{def:MartinComb}.
Generally the $\un\g$-sheet represents the domain $\Pi_{\un\iota}$ shifted by
$2\pi\eta(\g)$. Two sheets enumerated corresponding to  $\un \g$ and $\un{\g\g_j}$ are glued along the cuts with the common basis $2\pi\eta(\g)+\pi\eta_j=
2\pi\eta(\g\g_j)-\pi\eta_j$.
\end{remark}

Note that all character automorphic functions, $f$,  by 
$$
f_{\cR}(p)=f\circ \bl
$$ 
lift to single-valued functions on $\cR$. 
For typographical simplicity both functions will henceforth be denoted by $f$, but we will keep using this notation for the special functions 
$e_\cR(p,\a)$, $\l_\cR(p)$, $\mu_{\cR}(p)$ and $\theta_\cR(p)$, the functions on $\cR$ corresponding to $e_\a(\l), \l, \sqrt{\l}$ and $\Theta(\l)$, respectively.

Recall the character automorphic Hardy space $H^2(\a)$ introduced in Definition \ref{d:h2a}.   On this space, the linear functional of point evaluation in $\bbC_+$ is continuous.  Thus, by the Riesz representation theorem, there exist reproducing kernels $k^\a(p,p_0)=k^\a_{p_0}(p) \in H^2(\alpha)$ satisfying
\begin{align*}
	f(p_0)=\langle f,k^\a_{p_0}\rangle \quad\text{ for all }f\in H^2(\a).
\end{align*}

\begin{proposition}
	The reproducing kernels $k^\a(p,p_0)$ can be given by
	\begin{align}\label{eq:ReproducingKernels}
		k^\a(p,p_0)=i\frac{\mu_{\cR}(p)e_\cR(p,\a+\fj)\overline{e_\cR(p_0,\a)}+e_\cR(p,\a)\overline{\mu_\cR(p_0)e_\cR(p_0,\a + \fj)}}{\l_\cR(p)-\overline{\l_\cR(p_0)}}
	\end{align}
\end{proposition}
\begin{proof}
We note that the given vector belongs to $L^2$ with respect to the measure $\dd \theta_\cR$ and represents a function of Smirnov class. Therefore it belongs to $H^2(\a)$.
	Using that $\mu_\cR$ is real on $\partial \cR$ and \eqref{15may3} we see that, for $f^\a \in H^2(\a)$, one has
	\begin{align*}
		&\frac{1}{2\pi}\int_{\bbE}\overline{k^\a(p,p_0)}f^\a(p)\dd\t_{\cR}(p)\\
		&=\frac{1}{2\pi i}\oint_{E}\frac{\sqrt{\x}\,\tilde e_{\a+\fj}(\x)e_\a(\l_0)+\sqrt{\l_0}\,\tilde e_\a(\x)e_{\a+\fj}(\l_0)}{\cW(\x)(\x-\l_0)}\Theta'(\x) f^\a(p(\x))\dd \x.
	\end{align*}
	By DCT we obtain
	\begin{align*}
		\langle f^\a,k^\a_{p_0}\rangle =\frac{\sqrt{\l_0}\, \Theta'(\l_0) \left(\tilde e_{\a+\fj}(\l_0)e_\a(\l_0)+\tilde e_\a(\l_0)e_{\a+\fj}(\l_0)\right)}{\cW(\l_0)}f^\a(p_0).
	\end{align*}
	Thus by, \eqref{eq:WronskId} we have
	\begin{align*}
		\langle f^\a,k^\a_{p_0}\rangle =f^\a(p_0).
	\end{align*}
\end{proof}

We now clarify interrelations between the spaces $H^2(\a)$ and $H^2_{\G}(\a)\simeq H^2_{\cS_+}(\a)$. The scalar product in $H^2_{\cS_+}(\a)$ was defined by \eqref{18jan0} and also by Definition \ref{d18jan2} is clearly related to the harmonic measure on $E=\pd{\cS_+}$ w.r.t. the point $-1\in\cS_+$.
In other words this is a subspace of $L^2$ w.r.t. the measure $\dd \log\Phi(\l,-1)$. 

Recall that the Martin function $M(\l)$ is actually  defined up to a positive multiplier, its critical points were denoted by  $c_j\in(a_j,b_j)$. The critical points of the 
Green function $G(\l,-1)$ were denoted by  $\x_j(-1)\in(a_j,b_j)$.
\begin{lemma}\label{l19jan2}
The following ratio of two Abelian differentials is of the form
$$
\phi_{\cR}(p)=\frac{\dd \log \Phi(\l_{\cR}(p), -1)}{\dd \t_{\cR}(p)}=\frac{C}{1+\l_{\cR}(p)}\prod_{j\ge 1}\frac{1-\x_j(-1)/\l_{\cR}(p)}{1- c_j/\l_{\cR}(p)},\quad C>0.
$$
It is a function of bounded characteristic in $\cR$, moreover its inner part $\phi_{\cR}^{in}(p)$ is represented by the ratio of the Blaschke products, see \eqref{19jn1},
$$
\phi_{\cR}^{in}(p)=\frac 1{\Phi(\l_{\cR}(p),-1)}\prod_{j\ge 1}\frac{\Phi(\l_{\cR}(p),\x_j(-1))}{\Phi(\l_{\cR}(p), c_j)}.
$$
Let $\psi_{\cR}(p)=\sqrt{\phi_{\cR}(p)/\phi_{\cR}^{in}(p)}$ and $\a_{\psi}$ is the character generated by this function. Then 
$$
f(p)=\psi_{\cR}(p) g(\l_{\cR}(p))\in H^2(\a+\a_{\psi})
$$
if and only if $g\in H^2_{\cS_+}(\a)$ and
$
\|f\|_{H^2(\a+\a_{\psi})}=\|g\|_{H^2_{\cS_+}(\a)}
$. Respectively the reproducing kernels of the spaces are related by
\begin{equation}\label{21jan4}
k^{\a+\a_{\psi}}(p,p_0)=k^{\a}_{\cS_+}(\l_{\cR}(p),\l_{\cR}(p_0))\psi_{\cR}(p)\overline{\psi_{\cR}(p_0)}.
\end{equation}

\end{lemma}

\begin{proof}
By the definition $f(p)$ belongs to the Smirnov class, and its norm is finite  and coincides with the norm of $g$, just because both measures are mutually absolutely continuous. Conversely, we consider $\tilde g(z)$ as the lift of the ratio $f(p)/\psi(p)$ on the universal covering $\bbC_+$. Then $\tilde g$ is of Smirnov class in the upper half plane and it is square integrable w.r.t. to the harmonic measure on $\bbR$.
By the Smirnov maximum principle $\tilde g(z)$ belongs to $H^2$ and therefore possesses a harmonic majorant. Since $\tilde g(\g(z))=e^{2\pi i\a(\g)}\tilde g(z)$ we can interpret $\tilde g(z)=g(\bl(z))$, where $g\in H^2_{\cS_{+}}(\a)$.
\end{proof}

\subsection{From Fourier series to Fourier integral}\label{ss31}

Let $L_{+,x}^2=\{\hat f\in L^2(\bbR): \operatorname{supp}\hat f \subset (x,\infty)\}.$ In particular $L^2_+=L^2_{+,0}$.  The goal of this subsection is to prove the following theorem:
\begin{theorem}\label{thm:DY}
	Let $V\in \cV(E)$. Then there exists a Fourier transform $\cF=\cF(V)$
	\begin{align*}
		\cF: L^2\to L^2(\a),
	\end{align*}
	such that 
	\begin{align*}
		\cF L_V \hat f=\l \cF \hat f \quad\text{and}\quad
		\cF(L^2_{+,x})=e^{i\Theta(\l) x}H^2(\alpha-\eta x).
	\end{align*}
	For $\hat f\in L^2_{+,x}$ it is given  explicitly by
	\begin{align}\label{eq:FourierL2PlusAlpha}
		(\cF \hat f)(\l)=\int_{x}^{\infty}e^{i\Theta(\l)\xi}e_{\a-\eta \xi}(\l)\hat f(\xi)\dd \xi.
	\end{align}
	The corresponding Weyl-solution is of the form
	\begin{align}\label{eq:Weylsolution}
		u_+(x,\l)=\frac{e^{i\Theta(\l)x}e_{\a-\eta x}(\l)}{e_{\a}(\l)}.
	\end{align}
\end{theorem}

To the standard  spectral theory for 1-D Schr\"odinger operators we have to add the following Theorem \ref{thm:FourierIntegral}.
 See also  \cite{DY16}, where a relationship between spectral theorems for Jacobi matrices and 1-D Schr\"odinger operators was shown in the given context. In fact,  the resolvent $(L_V+1)^{-1}$ becomes  a Jacobi matrix with respect to the corresponding Fourier basis in the model space.
We get 
\eqref{eq:FourierL2PlusAlpha} as a limit  of the discrete Fourier representation found in \cite{SoYud94},
$$
(\cF_{SY} \hat f)(p)=\sum_{n=N}^\infty B_{\l_*}^n(\l_\cR(p)) \frac{k_{\l_*}^{\a-\nu_* n}(p)}{\|k_{\l_*}^{\a-\nu_* n}\|}\hat f_n, \quad \hat f\in l^2_{+,N},
$$
where $B_{\l_*} = \Phi(\cdot, \l_*)$ is the complex Green function w.r.t. $\l_*\in \cS_+$ and $\nu_*$ is the character  generated by this function.

\begin{theorem}\label{thm:FourierIntegral}
	The reproducing kernels $k^\a(p,p_0)$ of the Hardy spaces $H^2(\a)$ are given by
	\begin{align}\label{eq:ContFour}
		k^{\a}(p,p_0)-e^{i(\t_{\cR}(p)-\overline{\t_{\cR}(p_0)})x}k^{\a-\eta x}(p,p_0)=\int_{0}^{x}e^{i(\t_{\cR}(p)-\overline{\t_{\cR}(p_0)})\xi}e_\cR(p,\a-\eta \xi)\overline{e_\cR(p_0,\a-\eta \xi)}\dd \xi.
	\end{align}
\end{theorem}
\begin{lemma}\label{lem:nuisAC}
	Fix $p_0\in\cR$. The function 
	$$
	k^{\a-\eta x}(p_0,p_0)
	$$
	is Lipschitz continuous. In particular the measure 
	$$
	\dd e^{-2x} k^{\a-\eta x}(p_0,p_0)
	$$
	is absolutely continuous with respect to the Lebesgue measure. 
\end{lemma}
\begin{proof}
	Fix $\a\in\G^*$. Note that for a character automorphic inner function $w$ with character $\b$ we have,
	$$
	|w(p_0)k^{\a-\b}(p_0,p_0)|^2=|\langle wk^{\a-\b}_{p_0},k^{\a}_{p_0}\rangle|^2\leq k^{\a-\b}(p_0,p_0)k^{\a}(p_0,p_0).
	$$
	Taking $w(p) = e^{i\t_{\cR}(p)x}$ and noting that $k^{\a}(p_0, p_0)$ is real-valued, we have
	$$
	e^{i(\t_\cR(p_0)-\overline{\t_\cR(p_0)})x}k^{\a-\eta x}(p_0,p_0)\leq k^{\a}(p_0,p_0).
	$$
	Taking the log we see that
	$$
	\log k^{\a-\eta x}(p_0,p_0)-\log k^{\a}(p_0,p_0)\leq 2\Im\t_{\cR}(p_0)x.
	$$
	Our goal is to construct an inner function $e^{i\tilde \Theta(\l)}$ with inverse character $-\eta$. By the same trick we will then obtain 
	$$
	\log k^{\a}(p_0,p_0)-\log k^{\a-\eta x}(p_0,p_0)\leq 2\Im\tilde\t_{\cR}(p_0)x.
	$$
	Since we have proved this for arbitrary $\a\in\G^*$ and 
	$$
	0<\inf_{\a\in\G^*}k^\a(p_0,p_0)\le \sup_{\a\in\G^*}k^\a(p_0,p_0)<\infty,
	$$ 
	 this concludes the proof. 
	 
	 Note that the function
	$$
	H(\l)=\sqrt{\l}\prod_{j=1}^{\infty}\sqrt{\frac{\l-b_j}{\l-a_j}}
	$$ 
	has positive imaginary part in $\cS_+$. Due to the behaviour at infinity in the given domain, if we lift it to the universal cover, the corresponding Nevanlinna measure has a mass point at infinity on the universal covering. 
	We note also that in our case the measure corresponding to the lifting of the Martin function $M(\l)=\Im\Theta(\l)$ is pure point, see e.g.   \cite{VoYu16}. Thus, there is a constant $\vk> 0$ such that 
	$$
	\vk H(\bl(z))-\t(z)
	$$
	is represented by a positive measure, where $\t = \Theta \circ \bl$. Hence, we can set $\tilde\Theta:=\vk H-\Theta$, and, in fact, $\vk=1$. It  has positive  imaginary part,  and, since $H$ is single valued in $\cS_+$, $\tilde\Theta$ is additive character automorphic with the character $-\eta$. 
\end{proof}
Note that since our domain is Dirichlet regular and $M(-1) > 0$ we can choose a sequence $\{\l_N\}\subset\bbR_-$ such that $\l_N \to -\infty$ and
$$
G(-1,\l_N)=\frac{M(-1)}{N}.
$$
Let $p_* \in \cR$ be such that $\l_\cR(p_*) = -1$, and similarly let $p_N \in \cR$ be such that $\l_\cR(p_N) = \l_N$ defined above.  
For notational brevity, we henceforth denote $\l = \l_\cR(p)$ ($\l_0 = \l_\cR(p_0)$) unless otherwise noted. 

Since, for fixed $N$,
$$
B_{\l_N}(\l)^k\frac{k^{\a-k\nu_N}(p,p_N)}{\sqrt{k^{\a-k\nu_N}(p_N,p_N)}},
$$
form an orthonormal basis of $H^2(\a)$, we obtain that
\begin{align}\label{eq:DiscrFour}
	k^{\a}(p,p_0)-B_{\l_N}(\l)^n\overline{B_{\l_N}(\l_0)^n}k^{\a-n\nu_N}(p,p_0)\nonumber \\
	=\sum_{k=0}^{n-1}B_{\l_N}(\l)^k\overline{B_{\l_N}(\l_0)^k}\frac{k^{\a-n\nu_N}(p,p_N)\overline{k^{\a-n\nu_N}(p_0,p_N)}}{k^{\a-k\nu_N}(p_N,p_N)}.
\end{align}
We will show that the Fourier series \eqref{eq:DiscrFour} converges to the Fourier integral \eqref{eq:ContFour} as $N\to\infty$. 
\begin{lemma}\label{lem:cxGreenconv}
Suppose $n_N$ is such that $n_N/N \to x$ as $N \to \infty$ and $\l_N$ as above.  Then
\begin{align}
\label{eq:cxGreenconv}
\lim_{N \to \infty} B_{\l_N}(\l)^{n_N} = e^{i\Theta(\l)x}
\end{align}
and
\begin{align*}
\lim_{N \to \infty} n_N \nu_N = \eta x.
\end{align*}
\end{lemma}
\begin{proof}
This follows from the fact that
\begin{align*}
\lim_{N \to \infty} \frac{G(\l, \l_N)}{G(-1, \l_N)} &= \lim_{N\to\infty}\frac{-N\log|B_{\l_N}(\l)|}{M(-1)} \\
&= \frac{M(\l)}{M(-1)}
\end{align*}
and that $M(\l) = \Im(\Theta(\l))$.
\end{proof}

\begin{lemma}\label{eq:convergenceMeasure}
	Let 
	\begin{align*}
	f_N(k/N):=B_{\l_N}(-1)^{k}\frac{k^{\a-k\nu_N}(p_*,p_N)}{\sqrt{k^{\a-k\nu_N}(p_N,p_N)}}.
	\end{align*}
	and let $n_N$ be such that $n_N/N$ converges as $N \to \infty$. Then 
	\begin{align*}
	\lim\limits_{N\to\infty}\sum_{k=0}^{n_N-1}|f_N(k/N)|^2\d_{k/N}(A)\overset{\ast}{\rightharpoonup}\u(A),
	\end{align*}
	where
	\begin{align*}
	\dd\u(x)=-\dd(e^{-2\Im\Theta(-1)x}k^{\a-\eta x}(p_*,p_*))
	\end{align*}
	is absolutely continuous.
\end{lemma}
\begin{proof}
Evaluating \eqref{eq:DiscrFour} at $p = p_0 = p_*$, we have
$$
k^{\a}(p_*,p_*)-|B_{\l_N}(-1)|^{2n}k^{\a-n\nu_N}(p_*,p_*)=\sum_{k=0}^{n-1}|B_{\l_N}(-1)|^{2k}\frac{|k^{\a-k\nu_N}(p_*,p_N)|^2}{k^{\a-k\nu_N}(p_N,p_N)}.
$$
Let $x>0$ be such that $n_N/N\to x$ as $N\to\infty.$ Then we obtain for the right hand-side that 
\begin{align}\label{Dec07}
\lim_{N\to\infty}\left\{k^{\a}(p_*,p_*)-|B_{\l_N}(-1)|^{2n_N}k^{\a-n_N\nu_N}(p_*,p_*)\right\}\nonumber\\
=k^{\a}(p_*,p_*)-e^{-2\Im\Theta(-1)x}k^{\a-\eta x}(p_*,p_*).
\end{align}
On $\cB([0,x])$ we define the compact  family of  measures
$$
\u_N(A)=\sum_{k=0}^{n_N-1} |B_{\l_N}(-1)|^{2k}\frac{|k^{\a-k\nu_N}(p_*,p_N)|^2}{k^{\a-k\nu_N}(p_N,p_N)}\d_{k/N}(A).
$$
By \eqref{Dec07} we obtain that all subsequences converge to the same limit and hence
\begin{align*}
k^{\a}(p_*,p_*)-e^{-2\Im\Theta(-1)x}k^{\a-\eta x}(p_*,p_*)=\int_{0}^{x}\dd\u(\xi).
\end{align*}
Finally, $\u$ is absolutely continuous by   Lemma \ref{lem:nuisAC}. 
\end{proof}
\begin{proof}[Proof of Theorem \ref{thm:FourierIntegral}]
	In general we write
	\begin{align*}
	k^{\a}(p,p_0)&-B_{\l_N}(\l)^n\overline{B_{\l_N}(\l_0)^n}k^{\a-n\nu_N}(p,p_0)\\
	&=\sum_{k=0}^{n-1}B_{\l_N}(\l)^k\overline{B_{\l_N}(\l_0)}^k\frac{k^{\a-n\nu_N}(p,p_N)\overline{k^{\a-n\nu_N}(p_0,p_N)}}{k^{\a-k\nu_N}(p_N,p_N)}\\
	&=\sum_{k=0}^{n-1}\frac{B_{\l_N}(\l)^k\overline{B_{\l_N}(\l_0)^k}}{|B_{\l_N}(-1)|^{2k}}\frac{k^{\a-k\nu_N}(p,p_N)}{k^{\a-k\nu_N}(p_*,p_N)}\overline{\left(\frac{k^{\a-k\nu_N}(p_0,p_N)}{k^{\a-k\nu_N}(p_*,p_N)}\right)}|f_N(k/N)|^2
	\end{align*}
	Note that
	\begin{align*}
	k^{\a+\frac{\nu_N}{2}}(p,p_N)=C\sqrt{\frac{B_{\l_N}(\l)}{\l-\l_N}\prod_{j=1}^{\infty}\frac{\l-\l_j}{\l-c_j}\frac{B_{c_j}(\l)}{B_{\l_j}(\l)}}\prod_{j=1}^{\infty}B_{\l_j}(\l)^{\frac{1+\e_j}{2}},\quad C>0.
	\end{align*}
	That is,
	\begin{align*}
	\frac{k^{\a-k\nu_N}(p,p_N)}{k^{\a-k\nu_N}(p_*,p_N)}=\sqrt{\frac{B_{\l_N}(\l)}{B_{\l_N}(-1)}\frac{-1-\l_N}{\l-\l_N}}\frac{e_{\a-(k+1/2)\nu_N}(\l)}{e_{\a-(k+1/2)\nu_N}(-1)}
	\end{align*}
	Since $B_{\l_N}(-1)>0$,
	$$
	\lim\limits_{N\to\infty}\arg B_{\l_N}(\l)=0,
	$$
	that is 
	$$
	\lim\limits_{N\to\infty}B_{\l_N}(\l)=1.
	$$
	Hence, if $n_N/N\to x$ for some $x >0$ then 
	$$
	\lim_{N\to\infty}\frac{k^{\a-n_N\nu_N}(\l,\l_N)}{k^{\a-n_N\nu_N}(-1,\l_N)}=\frac{e_{\a-\eta x}(\l)}{e_{\a-\eta x}(-1)}
	$$
	We write
	\begin{align*}
	&\sum_{k=0}^{n-1}\frac{B_{\l_N}(\l)^k\overline{B_{\l_N}(\l_0)^k}}{|B_{\l_N}(-1)|^{2k}}\frac{k^{\a-k\nu_N}(p,p_N)}{k^{\a-k\nu_N}(p_*,p_N)}\overline{\frac{k^{\a-k\nu_N}(p_0,p_N)}{k^{\a-k\nu_N}(p_*,p_N)}}|f_N(k/N)|^2\\
	&=\sum_{k=0}^{n-1}\left(\frac{B_{\l_N}(\l)^k\overline{B_{\l_N}(\l_0)^k}}{|B_{\l_N}(-1)|^{2k}}\frac{k^{\a-k\nu_N}(p,p_N)}{k^{\a-k\nu_N}(p_*,p_N)}\overline{\left(\frac{k^{\a-k\nu_N}(p_0,p_N)}{k^{\a-k\nu_N}(p_*,p_N)}\right)}\right.\\
	&-\left.e^{i(\Theta(\l)-\overline{\Theta(\l_0)})k/N}e^{2\Im\Theta(-1)k/N}\frac{e_{\a-\eta k/N}(\l)}{e_{\a-\eta k/N}(-1)}\overline{\left(\frac{e_{\a-\eta k/N}(\l_0)}{e_{\a-\eta k/N}(-1)}\right)}\right)|f_N(k/N)|^2\\
	&+e^{i(\Theta(\l)-\overline{\Theta(\l_0)})k/N}e^{2\Im\Theta(-1)k/N}\frac{e_{\a-\eta k/N}(\l)}{e_{\a-\eta k/N}(-1)}\overline{\left(\frac{e_{\a-\eta k/N}(\l_0)}{e_{\a-\eta k/N}(-1)}\right)}|f_N(k/N)|^2
	\end{align*}
	Let us now again take a sequence $n_N/N$ increasing to $x$.
	We consider a function $g_N$ with 
	\begin{align*}
	g_N(k/N):=\frac{B_{\l_N}(\l)^k\overline{B_{\l_N}(\l_0)^k}}{|B_{\l_N}(-1)|^{2k}}\frac{k^{\a-k\nu_N}(p,p_N)}{k^{\a-k\nu_N}(p_*,p_N)}\overline{\left(\frac{k^{\a-k\nu_N}(p_0,p_N)}{k^{\a-k\nu_N}(p_*,p_N)}\right)}
	\end{align*}
	and linear in between. This family is equicontinuous, uniformly bounded, and converges pointwise for $k = n_N$ to the continuous function 
	$$
	e^{i(\Theta(\l)-\overline{\Theta(\l_0)}) x}e^{2\Im\Theta(-1)x}\frac{e_{\a-\eta x}(\l)}{e_{\a-\eta x}(-1)}\overline{\left(\frac{e_{\a-\eta x}(\l_0)}{e_{\a-\eta x}(-1)}\right)}.
	$$
	Hence, by Arzela-Ascoli it converges uniformly on $[0,x]$. Thus, the expression in brackets is less than $\e$ for $N$ sufficiently large and 
	$$
	\sum_{k=0}^{n_N-1}|f_N(k/N)|^2<C.
	$$
	By Lemma \ref{eq:convergenceMeasure} the second term converges to 
	$$
	\int_0^xe^{i(\Theta(\l)-\overline{\Theta(\l_0)})\xi}e^{2\Im\Theta(-1)\xi}\frac{e_{\a-\eta \xi}(\l)}{e_{\a-\eta \xi}(-1)}\overline{\left(\frac{e_{\a-\eta \xi}(\l_0)}{e_{\a-\eta \xi}(-1)}\right)}\dd\u(\xi).
	$$
	Thus, to conclude, we have shown that 
	\begin{align}\label{eq:FourierInt2}
			k^{\a}(p,p_0)-e^{i(\Theta(\l)-\overline{\Theta(\l_0)})x}k^{\a-\eta x}(p,p_0)=\int_{0}^{x}e^{i(\Theta(\l)-\overline{\Theta(\l_0)})\xi}e_{\a-\eta \xi}(\l)\overline{e_{\a-\eta \xi}(\l_0)}f_\a(\xi)\dd \xi,
	\end{align}
	with some integrable functions $f_\a$. It remains to show that $f_\a=1$ a.e.
	Recall that
	\begin{align*}
		k^{\a}(p,p_0)=e_{\a}(\l)\frac{m_+^\a(\l)-\overline{m_+^\a(\l_0)}}{\l-\overline{\l_0}}\overline{e_\a(\l_0)}.
	\end{align*}
	Differentiation \eqref{eq:FourierInt2} yields
	$$
	i(\Theta(\l)-\overline{\Theta(\l_0)})k^{\a-\eta x}(p,p_0)+\partial_xk^{\a-\eta x}(p,p_0)=-e_{\a-\eta x}(\l)\overline{e_{\a-\eta x}(\l_0)}f_\a(x).
	$$
	That is,
	\begin{align*}
		\partial_x\log k^{\a-\eta x}(p,p_0)=\frac{\Theta(\l)-\overline{\Theta(\l_0)}}{i}-\frac{\l-\overline{\l_0}}{m_+^{\a-\eta x}-\overline{m_+^{\a-\eta x}}}f_\a(x).
	\end{align*}
	Hence,
	\begin{align*}
		\log\frac{ k^{\a-\eta x}(p,p_0)}{ k^{\a}(p,p_0)}=\int_{0}^{x}\left\{\frac{\Theta(\l)-\overline{\Theta(\l_0)}}{i}-\frac{\l-\overline{\l_0}}{m_+^{\a-\eta \xi}(\l)-\overline{m_+^{\a-\eta \xi}(\l_0)}}f_\a(\xi)\right\}\dd\xi.
	\end{align*}
	Therefore, by using the expansion of $\Theta(\l)$ and $m_+^\beta(\l)$ as $\l\to-\infty$, we obtain that
	\begin{align*}
		\lim_{\l\to-\infty}\frac{\log\left|\frac{ k^{\a-\eta x}(p,p_0)}{ k^{\a}(p,p_0)}\right|}{M(\l)}=\int_0^x(1-f_\a(\xi))\dd\xi.
	\end{align*}
	In \cite{VoYu16} it is shown that 
	$$
			\limsup_{\l\to-\infty}\frac{-\log\left|B(\l)\right|}{M(\l)}=0,
	$$
	if $B$ is a Blaschke product, whose zeros are in $\bbR_+\setminus E$. We have,
	\begin{align*}
		\lim_{\l\to-\infty}\frac{\log\left|\frac{ k^{\a-\eta x}(p,p_0)}{ k^{\a}(p,p_0)}\right|}{M(\l)}&=\frac{1}{2}\lim_{\l\to-\infty}\frac{\log\left|\prod_{j=1}^{\infty}\frac{\l-\l_j}{\l-\tilde\l_j}\right|}{M(\l)}+\frac{1}{2}\lim_{\l\to-\infty}\frac{\log\left|\prod_{j=1}^{\infty}B_{\l_j}(\l)^{\e_j}B_{\tilde{\l_j}}^{\tilde\e_j}(\l)\right|}{M(\l)}\\
		&=\frac{1}{2}\lim_{\l\to-\infty}\frac{\log\left|\prod_{j=1}^{\infty}B_{\l_j}(\l)^{\e_j}B_{\tilde{\l_j}}^{\tilde\e_j}(\l)\right|}{M(\l)}=0
	\end{align*}
	Hence we conclude that for all $x>0$, 
	$$
	\int_0^x(1-f_\a(\xi))\dd\xi=0.
	$$
	The theorem is proved.
	\end{proof}

\subsection{$H^2_\cR$ as a shift invariant subspace of $L^2(L^2(\dd\a))$}

Recall the space $H^2_{\cR}$ defined in Definition \ref{d:h2r}.
Since $\cS_+$ is of Widom type there exists a measurable fundamental set $\bbE_0\subset\bbR$ for $\G'$. The set $\bbE_0$ for the action of $\G'$ is of course related to the fundamental set $\bbE$ for the action of $\G$ by
\begin{align*}
\bbE_0  = \bigcup_{\un\g\in\G/\G'} \un\g(\bbE).
\end{align*}  
Viewing $\cR$ as the quotient $\bbC_+/\Gamma'$, we can equivalently define $H^2_\cR$ as follows:
\begin{definition}
	The space $H^2_{\cR}$ is formed of those analytic functions $F$ in $\bbC_+$ such that:
	\begin{itemize}
		\item[(i)] $F$ is of Smirnov class,
		\item[(ii)] $F\circ \oc\g=F$ for all $\oc\g\in\G'$,
		\item[(iii)]  ${\displaystyle \frac 1 {2\pi}\int_{\bbE_0}|F(p)|^2\dd\theta_\cR(p)<\infty.}$ 
	\end{itemize}
\end{definition}
\begin{remark}
	Condition $(ii)$ means that we consider in fact single-valued functions on the Riemann surface $\cR$. For this reason we may also write $F(p)$ and 
	$$
	\frac 1 {2\pi}\int_{\partial\cR}|F(p)|^2\dd\theta_\cR(p).
	$$
\end{remark}
We denote the reproducing kernels of $H^2_{\cR}$ by $K(p,p_0)=K_{p_0}(p)$.  We have the following fundamental relationship between the space $H^2_\cR$ and the character automorphic Hardy spaces $H^2(\a)$ discussed above:
\begin{theorem}{\cite[Theorem 2.a]{Yud97}}
	$H^2_{\cR}=\int_{\G^*} H^2(\a)d\a$ in the following sense. Let $F(p)\in H^2_{\cR}$, then
	\begin{equation}\label{5may4}
	f(p,\alpha)=\sum_{\un\g\in \G/\G'} F\left(\g(p)\right)e^{-2\pi i\alpha(\un\g)}
	\end{equation}
	belongs to $L^2_{d\a}$ as a function on $\G^*$ and $f(p,\a)\in H^2(\a)$ as a function of $p$ for a.e. $\alpha$. Vice versa, if $f(p,\a)$ is a function with these properties then
	\begin{equation}\label{5may5}
	F(p)=\int_{\G^*} f(p,\a)\dd\a
	\end{equation}
	belongs to $H^2_{\cR}$. Moreover,
	\begin{equation}\label{11may1}
	\|F\|^2=\int_{\G^*}\|f(\cdot,\a)\|^2_{H^2(\a)}\dd\a.
	\end{equation}
	The reproducing kernels are related by
	\begin{align}\label{eq:lemReprod1}
	K(p,p_0)=\int_{\G^*}k^\alpha(p,p_0)\dd\alpha
	\end{align}
	and 
	\begin{align}\label{eq:lemReprod2}
	k^\alpha(p,p_0)=\sum_{\un\g\in\G/\G^*}K(\un\g(p),p_0)e^{-2\pi i\alpha(\un\g)}
	\end{align}
\end{theorem}

To a function $F\in H^2_{\cR}$ we associate the vector function
\begin{align}\label{eq:IsoImbedding}
	\hat f(\vt)=\{f_{\un\g}(\vt)\}_{\un \g\in\G/\G'},\quad\text{ where }f_{\un\g}(\vt)=F(p), \theta_{\cR}(p)=\vt, p\in\Pi_{\un \g}.
\end{align}
By $L^2(\ell_2(\G/\G'))$ we denote the space of $\ell_2(\G/\G')$ -valued functions $ f(\vt)\in\ell^2_{\G/\G'}$, with the norm
$$
\|\hat{f}\|^2=\frac 1{2\pi}\int_{\bbR}\|f(\vt)\|_{\ell^2_{\G/\G'}}^2\dd\vt.
$$
\begin{lemma}\label{lem:defInvSubspace}
	The mapping 
	\begin{align*}
	F\mapsto \hat f
	\end{align*}
	defined in \eqref{eq:IsoImbedding} maps $H^2_\cR$ isometrically into $L^2(\ell^2_{\G/\G'})$. That is,
	$$
	M:=\{\hat f: F\in H^2_\cR\}\subset L^2(\ell^2_{\G/\G'}).
	$$
\end{lemma}
\begin{proof}
	Since $\dd\t_\cR$ is invariant w.r.t. the action of the group $\G$, we obtain
	\begin{align*}
	\frac 1{2\pi}\int_{\bbE_0}|F(p)|^2\dd\t_{\cR}(p)=\frac 1{2\pi}\sum_{\un\g\in\G/\G'}\int_{\un\g(\bbE)}|F(p)|^2\dd\t_\cR(p)\\
	=\frac 1{2\pi}\sum_{\un\g\in\G/\G'}\int_{\bbR} |\hat f_{\un\g}(\vt)|^2\dd\vt=\frac 1{2\pi}\int_{\bbR}\|f_{\un\g}(\vt)\|^2_{\ell^2_{\G/\G'}}\dd\vt.
	\end{align*}
\end{proof}

Thus Lax-Halmos Theorem, see e.g. \cite[p. 17]{Nik} suggests an existence of the following representation
$H^2_{\cR}=\bTh H^2(\cE) $, where $\cE$ is a subspace of $\ell^2_{\G/\G'}$ and $\bTh$ a measurable operator valued function $\bbR$ whose values $\bTh(\xi)$ are isometric operators on $\cE$. Below we present an explicit form of such a representation (iii).
But, before to proceed we note that   $\ell^2_{\G/\G'}$ and $L^2_{\dd\a}$ are unitarily equivalent, see \eqref{5may3}. We will show that  one of them can be chosen as the scale space $\cE$.  For definitiveness, we denote $\cE=L^2_{d\a}$.

\begin{lemma}\label{lem:IsometryReprod}
		Let 
		$
		\mathfrak{V}:H^2_{\cR}\to H^2(\cE)
		$
		be defined by
		\begin{align*}
		(\mathfrak{V}K_{p_0})(\vt,\alpha)=\frac{i\overline{e_{\cR}(p_0,\a)}}{\vt-\overline{\vt_0}},
		\end{align*}
		where $\theta_{\cR}(p_0)=\vt_0$. Then $\mathfrak{V}$ defines an isometry on $H^2_{\cR}$. 
		\end{lemma}
\begin{proof}
		Due to \eqref{eq:lemReprod1}, \eqref{eq:ContFour},  Fubini's theorem and the shift-invariance of $\dd\alpha$ we see that
		\begin{align*}
		\langle K_{p_0},K_p\rangle_{H^2_\cR}&=K(p,p_0)\\
		&=\int_{\G^*}\int_0^\infty e^{i(\t_\cR(p)-\overline{\vt_0})x}e_{\alpha-\eta x}(\l_\cR(p))\overline{e_{\alpha-\eta x}(\l_\cR(p_0))}\dd x\dd\alpha\\
		&=\int_{\G^*}e_{\alpha}(\t_\cR(p))\overline{e_{\alpha}(\t_\cR(p_0))}\dd\alpha\int_0^\infty e^{i(\t_\cR(p)-\overline{\vt_0})x}\dd x\\
		&=\frac{i}{\t_\cR(p)-\overline{\vt_0}}\langle e_\alpha(\t_\cR(p)),e_\alpha(\t_\cR(p_0))\rangle_{L^2_{\dd\alpha}}\\
		&=\langle\frac{i}{\vt-\overline{\vt_0}},\frac{i}{\vt-\overline{\t_\cR(p)}}\rangle_{H^2}\langle e_\alpha(\t_\cR(p)),e_\alpha(\t_\cR(p_0))\rangle_{L^2_{\dd\alpha}}\\
		&=\langle \mathfrak{V}K_{p_0},\mathfrak{V}K_p\rangle_{H^2(\cE)}.
		\end{align*}
\end{proof}

\begin{itemize}
\item[(i)]
Recall that
\begin{equation}\label{4jan1}
\cF_1: H^2_{\cR}\to\int_{\G^*} H^2(\a)\dd\a\ \text{ such that }\ (\cF_1 F)(p,\a)=\sum_{\un\g\in\G/\G'} F(\un\g p)e^{-2\pi i\a(\g)}
\end{equation}
is a unitary operator.
\item[(ii)]
For $\cE=L^2_{d\a}$ we define $\cF_2:\int_{\G^*} H^2(\a)\dd\a\to L^2_+(\cE)$ making Fourier decomposition for any individual
$f(p,\a)\in H^2(\a)$, i.e.,
\begin{equation}\label{4jan2}
f(p,\a)=\int_0^{\infty}e^{i\t_{\cR}(p)x}e_{\cR}(p,\a-\eta x)\hat f(x,\a) \dd x, \quad \hat f(x,a)\in L^2_+.
\end{equation}
\end{itemize}
We note that the operator $\hat g(x,\a)=(\cU_\eta \hat f)(x,\a):=\hat f(x,\a+\eta x)$ acts unitary in $L_+^2(\cE)$. Thus we can define an alternative representation
\begin{equation}\label{4jan3}
g(p,\a)=\int_0^{\infty}e^{i\t_{\cR}(p)x}e_{\cR}(p,\a-\eta x)\hat g(x,\a-\eta x) \dd x, \quad \hat g(x,\a)\in L^2_+(\cE).
\end{equation}
\begin{itemize}
\item[(iii)]
We get the unitary operator $\cF_3: H^2(\cE)\to H^2_{\cR}$ acting as
\begin{equation}\label{4jan30}
(\cF_3 g)(p)=\int_{\G^*}e_{\cR}(p,\a)g(\t_{\cR}(p),\a)\dd\a, \quad g(\vt,\a)=\int_{0}^\infty e^{i\vt x} \hat g(x,\a) \dd x.
\end{equation}
\end{itemize}


\begin{remark} We claim that for $\Im\vt_0>0$
\begin{equation}\label{12jan3}
L^2_{d\a}=\cE=\clos\left\{\sum_{\# \{\un\g\}<\infty} C_{\un\g} \overline{e_{\cR}(p_{\un\g},\a)}: \t_{\cR}(p_{\un\g})=\vt_0\right\}.
\end{equation}
Indeed,
$\frac{\vt-\vt_0}{\vt-\overline{\vt_0}} H^2(\cE)$ is mapped unitary on $w_0 H^2_{\cR}$, where $w_0(p)=\frac{\t_{\cR}(p)-\vt_0}{\t_{\cR}(p)-\overline{\vt_0}}$. It is clear that 
$$
 H^2(\cE)\ominus\frac{\vt-\vt_0}{\vt-\overline{\vt_0}} H^2(\cE)=\left\{\frac{f}{\vt-\overline{\vt_0}}:\ f\in \cE\right\}.
$$
On the other hand, since $w_0(p)$ is a Blaschke product
$$
 H^2_{\cR}\ominus w_0 H^2_{\cR}=\clos\left\{\sum_{\# \{\un\g\}<\infty} C_{\un\g} \cK_{p_{\un\g}}(p): \t_{\cR}(p_{\un\g})=\vt_0\right\}.
$$
Since they are unitary equivalent, by Lemma \ref{lem:IsometryReprod}, we have \eqref{12jan3}.

\end{remark}

\section{Generalized Abelian integrals and the KdV hierarchy}

It is well known that in the finite gap setting the direction of the time shift is generated by Abelian integrals of the second kind, see e.g. \cite{GeHo03}. The following functions will serve as \textit{generalized Abelian integrals} on $\cR$; cf. \cite[Theorem 5]{VoYu16} for the case $M=M_0$.
\begin{proposition}
\label{pr:genabint}
	Let $k\in\bbN$, $v_k(\l)=\Im\l^{k+\frac{1}{2}}$ and assume that 
	\begin{align}
	\label{eq:halfmomentcondn}
	\int_{\bbR_+\setminus E}G(i,\xi)\dd\xi^{k+\frac{1}{2}}<\infty.
	\end{align}
	Then
	\begin{align*}
		M_k(\l)=v_k(\l)+\frac{1}{\pi}\int_{\bbR_+\setminus E}G(\l,\xi)\dd\xi^{k+\frac{1}{2}}
	\end{align*}
	defines a harmonic function in $\cS_+$. 
\end{proposition}
\begin{proof}
	Clearly $M_k(\l)$ is harmonic in $\bbC\setminus\bbR_+$. On the gaps $(a_j,b_j)$ we have that $v_k(\l+i0)=v_k(\l-i0)$ but the derivative 
	$
	\frac{\partial v_k}{\partial y}
	$
	has a jump. Due to the Cauchy Riemann equation we find that the generalized Laplacian of $v_k$ is given by
	$$
	\Delta v_k(\xi)=2\chi_{\bbR_+\setminus E}(\xi)\dd\xi^{k+\frac{1}{2}}.
	$$
	 Hence, $M_k$ is harmonic in $\cS_+$. 
\end{proof}

\begin{remark}
Note that the condition \eqref{eq:halfmomentcondn} follows immediately from the assumptions (PW) and ($k$-GLC).
\end{remark}

Note that in particular $M_0(\l)$ is a positive harmonic function which, since we assumed that $E$ is Dirichlet regular, vanishes on the boundary. Hence, $M_0 = M$ defines the Martin function of $\cS_+$.  Let $\Theta_k$ define the analytic function in the domain such that $\Im \Theta_k=M_k$. 
\begin{definition}
	Let $\Theta_k$ be defined as above. We define the generalized normalized Abelian integral of order $k$ as the function on $\cR$ of the form 
	$$
	\t_{\cR}^{(k)}=\Theta_k\circ \bl.
	$$
	By $\eta^{(k)}$ we denote the additive character generated by this function, i.e.,
	\begin{equation}\label{16jan1}
\t_{\cR}^{(k)}(\un\g p)=\t_{\cR}^{(k)}(p)+2\pi\eta^{(k)}(\g), \quad \g\in\G.
\end{equation}
\end{definition}
\begin{remark}
Let $\G_0$ be a normal subgroup of $\G$ which contains $\G'$.  The Riemann surface $\cR=\cR_{\G_0}=\bbC_+/\G_0$ is also an abelian covering of $\O=\bbC\setminus E$. The group $\G/\G_0$ acts on this surface.  Further, if $\a_0\in\G^*$ we can define $\oc\a_0=\a_0|\G_0$. Evidently, the collection  
$\Xi_0:=\{\a\in\G^*:\ \a|\G_0=\oc\a_0\}$ coincides with $(\G/\G_0)^*$ up to a shift by $\a_0$. Thus $\dd\a$ can be naturally defined on $\Xi_0$ by
the Haar measure on $(\G/\G_0)^*$. Particularly we are interested in 
$$
\G_0=\G_0(\eta, \eta^{(k)})=\{\g\in\G:\ \eta^{(k)}(\g)=0,\ \eta(\g)=0\}.
$$
$\t_{\cR}^{(k)}(p)$ can be treated as a function on this surface with the group action \eqref{16jan1}.  Moreover, the two dimensional  flow
$$
\a\mapsto\a -x\eta-t \eta^{(k)},\quad \a\in\Xi_0,
$$
in this case, is ergodic with respect to $\dd \a$.
\end{remark}

\subsection{Relation to the finite-gap case}
Let 
$$
E_N=\bbR_+\setminus \cup_{j=1}^N(a_j,b_j), \quad \cS_{+,N}=\bbC\setminus E_N
$$
It is convenient to understand the Hardy space $H^2_{\cS_+}(\a)$  as a space of multivalued functions  having harmonic majorant in the domain, see Introduction and Lemma \ref{l19jan2}.
\begin{lemma}
Let $\a_N=\a|\pi_1(\cS_{+,N})$ and $ k_{\cS_{+,N}}^{\a_N}(\l,\l_0)$ be the reproducing kernel in  $H^2_{\cS_{+,N}}(\a_N)$. Then
$$
k_{\cS_{+,N}}^{\a_N}(\l,\l_0)\to k_{\cS_+}^\a(\l,\l_0)
$$
on compact subsets in $\cS_+$.
\end{lemma}
\begin{proof}
$k_{\cS_+}^\a(\l,\l_0)$ has a harmonic majorant in $\cS_{+,N}$ and therefore belongs to $H^2_{\cS_{+,N}}(\a_N)$, moreover
$$
\|k_{\cS_+,\l_0}^\a\|^2_{H^2_{\cS_{+,N}}(\a_N)}\le \|k_{\cS_+,\l_0}^\a\|^2_{H^2_{\cS_+}(\a)}=k_{\cS_+}^\a(\l_0,\l_0).
$$
We have
\begin{align}
0\le\| k_{\cS_+,\l_0}^\a-k_{\cS_{+,N},\l_0}^{\a_N}\|_{H^2_{\cS_{+,N}}(\a_N)}^2=k_{\cS_{+,N}}^{\a_N}(\l_0,\l_0)-2k_{\cS_+}^\a(\l_0,\l_0)+\|k_{\cS_+,\l_0}^\a\|^2_{H^2_{\cS_{+,N}}(\a_N)}
\nonumber\\
\le k_{\cS_{+,N}}^{\a_N}(\l_0,\l_0)-k_{\cS_+}^{\a}(\l_0,\l_0). \label{11jan1}
\end{align}

On the other hand the family $\{k_{\cS_{+,N}}^{\a_N}(\bl^N(z),\l_0)\}_{N}$ is compact in the standard $H^2$ (w.r.t. the harmonic measure). We choose a subsequence $N_j$ so that
$$
f(z)=\lim_{N_j\to\infty} k_{\cS_{+,N_j}}^{\a_{N_j}}(\bl^{N_j}(z),\l_0).
$$
We note that this function can be understood as an element of $H^2_{\cS_+}(\a)$, i.e., $f(z)=g(\bl(z))$, where $g\in H^2_{\cS_+}(\a)$ and moreover
$$
\|f\|^2_{H^2}=\| g\|^2_{H^2_{\cS_+}(\a)}\le \lim_{N_j\to\infty} \| k_{\cS_{+,N_j},\l_0}^{\a_{N_j}}\circ\bl^{N_j}\|_{H^2_{\cS_{+,N_j}}(\a_{N_j})}^2=\lim_{N_j\to\infty} k_{\cS_{+,N_j}}^{\a_{N_j}}(\l_0,\l_0).
$$
Thus
\begin{equation}\label{11jan2}
\left(\lim_{N_j\to\infty}  k_{\cS_{+,N_j}}^{\a_{N_j}}(\l_0,\l_0)\right)^2=|g(\l_0)|^2\le k_{\cS_+}^\a(\l_0,\l_0)\|g\|_{H^2_{\cS_+}(\a)}^2\le k_{\cS_+}^\a(\l_0\l_0)\lim_{N_j\to\infty}  k_{\cS_{+,N_j}}^{\a_{N_j}}(\l_0,\l_0)
\end{equation}
As a combination of \eqref{11jan1} and \eqref{11jan2} we have
$$
\lim_{N_j\to\infty}  k_{\cS_{+,N_j}}^{\a_{N_j}}(\l_0,\l_0)=k_{\cS_+}^\a(\l_0,\l_0).
$$
That is, in fact, the sequence $\{k_{\cS_{+,N}}^{\a_N}(\bl^N(z),\l_0)\}_{N}$ converges to $k_{\cS_+}^\a(\bl(z),\l_0)$ in $H^2$.
\end{proof}
\begin{corollary}
Uniformly on compact subsets in $\cS_+$
$$
e_{\a_N}(\l)\to e_{\a}(\l)
$$
\end{corollary}

We wish to explain further the relationship between the generalized Abelian integrals $\theta_{\cR}^{(k)}$ and the typical Abelian integrals on a hyperelliptic Riemann surface.  Consider the family of polynomials
\begin{align*}
s_N(\lambda) := \l\prod_{j = 1}^N (\lambda - a_j)(\lambda - b_j)
\end{align*}
and the associated family of hyperelliptic Riemann surfaces
\begin{align*}
\cS_N := \{(\l, w) : w^2 = s_N(\l)\}.
\end{align*}
We fix a basis $\{A_j, B_j\}_{j = 1}^N$ for the homology of $\cS_N$: let $A_j$ denote the equivalence class of loops forming a clockwise circle around $[a_j, b_j]$ on the upper sheet $\cS_{+,N}$, and let $B_j$ denote the equivalence class of loops beginning at $-1$, passing through the gap $(a_j, b_j)$ from the upper sheet to the lower sheet, and then returning to $-1$. 

Consider now the usual Abelian integrals of the second kind $\omega_2^{k,N}$ with pole at $\infty$ of order $k$ in $\e_N$, written in local coordinates near $P = \infty$ as
\begin{align*}
\dd \omega_2^{k,N}(\e_N) \sim \left(\frac{1}{\e_N^{k}} + f_{k,N}(\e_N)\right)\dd \e_N,
\end{align*}
where $f_{k,N}$ is holomorphic and $k \geq 2$.  This form is unique up to the addition of Abelian differentials of the first kind; we normalize by assuming $f_{k,N}(\e)\dd \e$ has vanishing $A$-periods.

We denote by $\Theta_N^{(k)}$ the generalized Abelian integrals corresponding to $\Omega_N$; that is,
\begin{align*}
M_{k,N}(\lambda) &:= \Im(\lambda^{k+\frac{1}{2}}) + \frac{1}{\pi}\int_{\bbR_+ \setminus E_N} G_N(\lambda, \x)\dd \x^{k+\frac{1}{2}}, \\
\Im(\Theta^{(k)}_N(\lambda)) &= M_{k,N}(\lambda)
\end{align*}
By the additive automorphic property of $\Theta^{(k)}_N$, $\dd \Theta^{(k)}_N$ is an admissible differential form on the Riemann surface $\cS_N$ with pole only at the point $\infty$, where $\Theta_N^{(k)}$ has asymptotic behavior $\Theta^{(k)}_N(\lambda) \sim \lambda^{k+\frac{1}{2}}$.  Expanding at $\infty$ in the coordinate $\e = \frac{1}{\sqrt{\lambda}}$, one has
\begin{align*}
\dd \Theta_N^{(k)}(\e) &\sim -(2k+1)\frac{1}{\e^{2k+2}} \dd \e.
\end{align*}
Consider the differential form
\begin{align*}
\dd \left(\Theta_N^{(k)} + (2k+1)\omega_2^{2k+2,N}\right)
\end{align*}
which is given in local coordinates at $\infty$ by $(2k+1)f_{2k+2,N}(\e)\dd\e$.  This is an Abelian differential of the first kind, and is thus determined by its period class; since $f_{2k+2,N}(\e)\dd \e$ has vanishing $A$-periods, we conclude that the $B$-periods of $\dd \Theta_N^{(k)}$ and $-(2k+1)\dd\omega_2^{2k+2,N}$ must agree.

Denote by $\omega_1^{l,N}$, $1 \leq l \leq N$ denote a basis of Abelian integrals of the first kind on $\cS_N$, normalized such that
\begin{align*}
\int_{A_j} \dd \omega_1^{l,N} = \delta_{j,l}.
\end{align*}
We can explicitly compute the $B$-periods of $ \dd \Theta_N^{(k)}$ in terms of these integrals:
\begin{proposition}
The $B$-periods of $\dd\Theta_N^{(k)}$ are given by
\begin{align*}
\int_{B_j} \dd \Theta_N^{(k)} &= -2\pi i \frac{1}{(2k)!}\partial^{2k+1}\omega_1^{j,N}|_{\infty}
\end{align*}
\end{proposition}
\begin{proof}
We can represent $\cS_N$ as a regular $4N$-gon $R$ by cutting along representative loops $(A_j, B_j)$ and identifying the corresponding sides.  Using the residue theorem and the normalizations of $\dd \omega_1^{l,N}$ and $\dd \omega_2^{2k+2,N}$, we have
\begin{align*}
\int_{B_j} \dd \omega_2^{2k+2,N} &= \int_{\partial R} \omega_1^{j,N} \dd \omega_2^{2k+2,N} \\
&= 2\pi i \Res_\infty(\omega_1^{j,N}\dd \omega_2^{2k+2,N}) \\
&= 2\pi i \frac{1}{(2k+1)!}\partial^{2k+1}\omega_1^{j,N}|_{\infty}.
\end{align*}
Because the $B$-periods of $\dd\Theta_N^{(k)}$ are $-(2k+1)$ times the $B$-periods of $\dd \omega_2^{2k+2,N}$, the proposition follows.
\end{proof}

It is not hard to see that these $M_{k,N}$ converge pointwise to the generalized Abelian integrals $M_k$:
\begin{lemma}
Suppose \eqref{eq:halfmomentcondn} holds.  Then for all $\lambda \in \cS_+$,
\begin{align*}
\lim_{N \to \infty} M_{k,N}(\lambda) = M_k(\lambda),
\end{align*}
where we take the limit for $N \geq N_0$ such that $\lambda \in \cS_{+,N_0}$.
\end{lemma}
\begin{proof}
Let $N_0$ be such that $\lambda \in \cS_{+,N_0}$.  By monotonicity of the domains, we have that for $N \geq N_0$,
\begin{align*}
M_{k}(\lambda) - M_{k,N}(\lambda) &= \frac{1}{\pi}\left(\int_{\bbR_+ \setminus E_N} (G(\lambda, \x) - G_N(\lambda,\x)) \dd \x^{k+\frac{1}{2}} + \sum_{j > N} \int_{a_j}^{b_j} G(\lambda, \x) \dd \x^{k+\frac{1}{2}}\right)
\end{align*}
and by the monotone convergence theorem and condition \eqref{eq:halfmomentcondn} $M_{k}(\lambda) - M_{k,N}(\lambda)$ tends to $0$.
\end{proof}

Denote by $\eta_{k,N}$ the additive characters of the function $\Theta_N^{(k)}$, and let $\widetilde{\eta}_{k,N} \in \Gamma^*$ be the character formed by including $\eta_{k,N}$ into the larger group $\Gamma^*$ by
\begin{align*}
\widetilde{\eta}_{k,N}(\gamma) = \begin{cases}
\eta_{k,N}(\gamma) & \gamma \in \pi_1(\Omega_N) \\
0 & \text{otherwise}
\end{cases}.
\end{align*}
\begin{corollary}
Suppose \eqref{eq:halfmomentcondn} holds.  Then the characters $\widetilde{\eta}_{k,N}$ converge to $\eta^{(k)}$.
\end{corollary}

With these facts in hand, we now have the tools to prove the following

\begin{theorem} There exist polynomials
$$
A_k(\l,\a)=
\l^{k}+\cA_1(\a)\l^{k-1}+\cA_k(\a), \quad \text{and}
$$
$$
B_k(\l,\a)
=B_0(\a)\l^{k}+B_1(\a)\l^{k-1}+B_k(\a)
$$
such  that for the generalized Abelian integral $\t_{\cR}^{(k)}(p)$ we have
\begin{equation}\label{p17oct}
(\t_{\cR}^{(k)}(p)+i\pd_{\eta^{(k)}})e_{\cR}(p,\a)
=A_k(\l_{\cR}(p),\a) \mu_{\cR}(p)e_{\cR}(p,\a+\fj)-B_k(\l_{\cR}(p),\a)
 e_{\cR}(p,\a).
\end{equation}
\end{theorem}

\begin{proof}
For a fixed $N$ and $t_k > 0$, we write the corresponding relation to \eqref{p17oct} in the integral form:
\begin{align*}
-i(e^N_{\a_N - t_k\eta_N^{(k)}}(\l) - e^N_{\a_N}(\l)) =
\int_0^{t_k}\left(A_{k,N}(\l,\a_N - \x \eta_N^{(k)})\sqrt{\l}e^N_{\a_N + \fj - \x \eta_N^{(k)}}(\l) \right. \\ \left. - B_{k,N}(\l,\a_N - \x \eta_N^{(k)})e^N_{\a_N - \x \eta_N^{(k)}}(\l) - \Theta^{(k)}_N(\l)e^N_{\a_N - \x \eta_N^{(k)}}(\l) \right) \dd \x.
\end{align*}
This representation is classical in the finite gap case (cf. e.g. \cite{GeHo03}).  Then we pass to the limit as $N\to \infty$, using the discussion of convergence above. A compactness argument shows that the representation \eqref{p17oct} exists.
\end{proof}

\begin{remark}Let $\ve(\l,\a)\in H^\infty_{\cS_+}(\a)$ be  the extremal function in the following sense
$$
\ve(-1,\a)=\sup\{w(-1): w\in H^\infty_{\cS_+}(\a),\ \|w\|_\infty\le 1\}, \ \text{where}\ \|\ve(-1,\a)\|_\infty\le 1, \a\in\G^*.
$$
By the DCT $\ve(\l,\a)\to 1$ as $\a\to 0_{\G^*}$ \cite[(21) p. 205]{Has83} uniformly on  compact subsets in $\cS_+$.
Despite the similar notation, we emphasize that it is not necessarily the case that $\eta_{k,N}$ agree with the restriction $\eta^{(k)}_N = \eta^{(k)}|\pi_1(\cS_{+,N})$. For this reason in the approximation procedure above one has to consider the correction functions $\ve(\l,(\tilde\eta_{k,N}-\eta_N^{(k)})\xi)$, $\xi\in [0,1]$. 
\end{remark}

\subsection{An explicit map $\Gamma^* \to \cV(E)$}

We wish to recover the potential function $V$ and its derivatives not via the traditional trace formulas, but rather by way of asymptotic expansion of the $m$-function.  Recall that (DCT) guarantees that $m_+^\a(\l_0)$ is continuous on $\G^*$ for internal points $\l_0\in \cS_+$. We will show that this also holds for $\l_0=0$. As an immediate consequence, we will obtain an explicit expression for $V(x)$ in terms of our special functions on $\G^*$.
\begin{theorem}
\label{t:chi0diff}
	Let $\eta\in\G^*$ be the direction corresponding to the function $\Theta$; cf. \eqref{5amay3}. Define
	\begin{align}\label{def:chi0}
	\chi_0(\a)= im_+^\a(0).
	\end{align}
	Then $\chi_0$ is continuous on $\G^*$. Moreover, $\chi_0$ is differentiable in direction $\eta$. Defining
	\begin{align}\label{def:chi1}
		\chi_1(\a)=\frac{1}{2}(\chi_0(\a)^2 - i\partial_{\eta}\chi_0(\a)).
	\end{align}
	 the potential $V$ corresponding to $m_+^\a$ is given by
	\begin{align}\label{eq:potentialSpecialFunc}
		V(x)=-2\chi_1(\a-\eta x).
	\end{align}
\end{theorem}
\begin{proof}
	First we prove continuity of $\chi_0$. Since for fixed $\a$, $m_\pm^\a(0)$ exists and $m_\pm^\a$ is increasing on $\bbR_-$, for fixed $\e>0$ there exist $\x<0$ such that
	\begin{align*}
		m_\pm^\a(0)-\e\leq m_\pm^\a(\x).
	\end{align*}
	Let $\{\b_j\}$ such that $\lim\b_j=\a$ and $\{k_j\}$ a subsequence such that 
	\begin{align*}
		\lim\limits_{j\to\infty}m_\pm^{\b_{k_j}}(0)=\liminf_{j\to\infty}m_\pm^{\b_{j}}(0)
	\end{align*}
	By continuity of $m_\pm^\a(\x)$, there exists $j_0$ such that for all $j\geq j_0$ we have
	\begin{align*}
		m_\pm^{\b_{k_j}}(0)\geq 	m_\pm^{\b_{k_j}}(\x)\geq 	m_\pm^{\a}(\x)-\e\geq m_\pm^\a(0)-2\e.
	\end{align*}
	That is,
	\begin{align*}
		\liminf_{j\to\infty}m_\pm^{\b_{k_j}}(0)\geq m_\pm^\a(0)
	\end{align*}
	Using that $m_+^\a(0)=-m_-^\a(0)$ also get
	\begin{align*}
		\limsup_{j\to\infty}m_+^{\b_j}(0)\leq m_+^\a(0),
	\end{align*}
	which proves continuity of $\chi(\a)$. Let $\x<0$. Due to Theorem \ref{thm:DY} 
	\begin{align*}
		m_+^\a(x,\x)=m^{\a-\eta x}(\x).
	\end{align*} 
	Recall the Riccati equation for $m_+^\a$,
	\begin{align*}
		\partial_x m_+^{\a-\eta x}(\l)=V(x)-\l-(m_+^{\a-\eta x}(\l))^2.
	\end{align*}
	That is,
	\begin{align*}
		m_+^{\a-\eta x}(\l)-m_+^{\a}(\l)=\int_0^x\{V(\xi)-\l-(m_+^{\a-\eta \xi}(\l))^2\}\dd\xi.
	\end{align*}
	Since $m_+^\a(\l)$ is continuous on $\G^*$ and $m_+^\a(\l)$ has no pole on $\bbR_-$ we can apply the dominated convergence theorem and pass to the limit to obtain
	\begin{align*}
			m_+^{\a-\eta x}(0)-m_+^{\a}(0)=\int_0^x(V(\xi)-(m_+^{\a-\eta \xi}(0))^2)\dd\xi,
	\end{align*}
	which proves differentiability of $\chi_0$ and \eqref{eq:potentialSpecialFunc}. 
\end{proof}

\subsection{(PW$_{M}$) condition and asymptotics for $e_{\cR}(p,\a)$}\label{ss43}

As the foremost consequence of the  condition (PW$_M$)
\begin{equation}\label{25jan1}
\fw_M:=\sum_{\nabla M(c)=0} M(c)<\infty
\end{equation}
we show asymptotics not only for the ratio
$m^\a_+(\l)$ but for the forming it functions $e_{\cR}(p,\a)$ themself,
see Corollary \ref{p30jan1}. Similarly, we get asymptotics for the generalized Abelian integrals, see Corollary \ref{p30jan4}.
Both are based on the following lemma.
\begin{lemma}
Let $\l_*<-1$, where $-1$ is a normalization point. Then
\begin{equation}\label{25jan2}
\frac{M(\l)}{M(-1)}\ge \frac{G(\l,\l_*)}{G(-1,\l_*)}
\end{equation}
where $\l$ belongs to an arbitrary gap $(a_j,b_j)$, $j\ge 1$.
\end{lemma}

\begin{proof}
Let $E_N = \bbR_+ \setminus \cup_{j \leq N} (a_j,b_j)$, and consider the Abelian differential on $\cS_{+,N} = \bbC \setminus E_N$
$$
\dd\O(\l)=\dd\T_N(\l)+\vk \dd\log\Phi_N(\l,\l_*)=\frac{\pi(\l)\dd\l}{(\l-\l_*)\sqrt{s_N(\l)}}, \quad \vk=\frac{M_N(-1)}{G_N(-1,\l_*)}.
$$
Here $\pi(\l)$ is a certain monic polynomial of degree $N+1$. We can localize all critical points of the given Abelian differential, i.e. the zeros of $\pi(\l)$. Indeed, by the normalization condition 
$$
\int_{a_j}^{b_j}\dd\O(\l)=0,
$$
within each gap we have at least one zero of $\pi(\l)$. By the normalization $\O(\l)=\int_0^\l d\O$ we have $\O(0)=0$. But due to the definition of $\vk$ we have $\O(-1)=0$. Thus the interval $(-1,0)$ also contains a critical point. Since $\deg\pi=N+1$ we listed positions of all critical points: each gap and the interval $(-1,0)$ contains exactly one simple critical point. Thus $\O$ is an integral of Schwarz-Christoffel type. It maps conformally the upper half-plane on a domain sketched below in Fig. \ref{fig1}.
We get
$$
\frac{M_N(\l)}{M_N(-1)}- \frac{G_N(\l,\l_*)}{G_N(-1,\l_*)}= \frac{1}{M_N(-1)}\Im \O(\l)\ge 0, \quad \l\in (a_j,b_j).
$$
Passing to the limit as $N\to \infty$ for fixed $\l\in(a_j,b_j)$ and $\l_*<1$, we get \eqref{25jan2}.
\end{proof}
\begin{figure}[htbp] 
\begin{center}
\includegraphics[scale=0.3]
{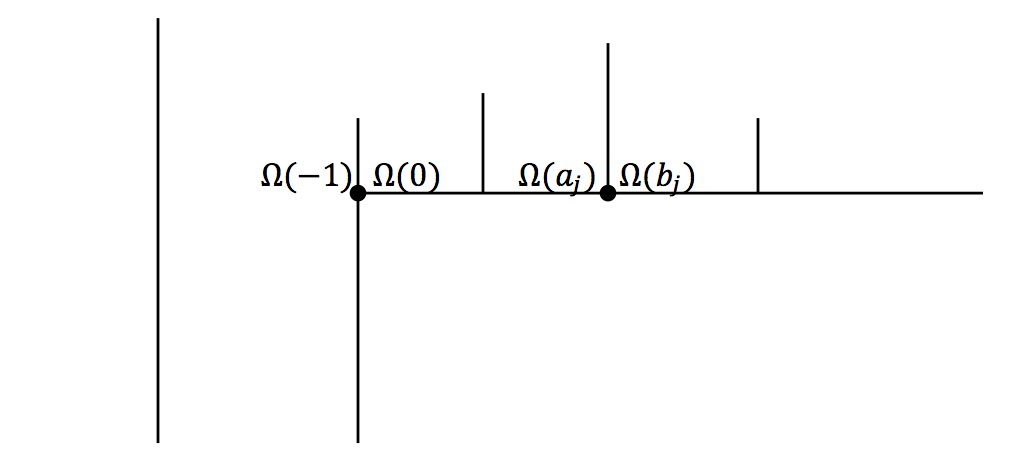}
\caption{Comb-domain $\O(\bbC_+)$. $\O(0)=\O(-1)$, $\O(a_j)=\O(b_j)$, $j\ge 1$.}\label{fig1}
\label{comb}
\end{center}
\end{figure}

As a consequence, we have the following corollary:
\begin{corollary}\label{c:qwtowid}
The condition \eqref{25jan1} implies the Widom condition (1.11).
\end{corollary}

Moreover, we can explicitly compute the second term in asymptotics.

\begin{theorem} Assume that \eqref{25jan1} holds. Let $D=\{(\l_j,\e_j)\}\in\cD(E)$, to which we associate 
the Blaschke product
$$
\Phi_D(\l)=\prod_{j\ge 1}\Phi (\l,\l_j)^{\frac{1+\e_j}2},
$$
assuming the normalizations $\Phi_{\l_j}(\l)\ge 0$ in $\bbR_-$. Then the following
limit exists
\begin{equation}\label{25jan3}
\lim_{\l\to-\infty}\frac {M(-1)}{G(\l,-1)}\log\frac{1}{\Phi_D(\l)}=\sum_{j\ge 1}\frac{1+\e_j}{2}M(\l_j)
\end{equation}
\end{theorem}

\begin{proof}
 For a fixed $\l<-1$,  we have
\begin{equation}\label{25jan4}
\frac 1{G(\l,-1)}
\log\frac 1{\Phi_D(\l)}= \sum_{j\ge1}\frac{1+\e_j} 2\frac{G(\l_j,\l)}{G(-1,\l)}.
\end{equation}
Recall that 
$$
\lim_{\l\to-\infty} \frac{G(\l_j,\l)}{G(-1,\l)}=\frac{M(\l_j)}{M(-1)}.
$$
Note that $M(c_j)\ge M(\l_j)$ for an arbitrary $\l_j\in(a_j,b_j)$. By \eqref{25jan2} and \eqref{25jan1} the sequence $\left\{\frac{M(c_j)}{M(-1)}\right\}_{j\ge 1}$ form an integrable majorant . Therefore we can pass to the limit in \eqref{25jan4} and we obtain \eqref{25jan3}.
\end{proof}

\begin{corollary}\label{p30jan1}
Assume that \eqref{25jan1} holds and
\begin{equation}\label{25jan11}
\int_{\bbC_+\setminus E} \dd\xi^{1/2}<\infty.
\end{equation}
Then 
\begin{equation}\label{25jan12}
\lim_{y\to\infty}y \log e_{\cR}(p(iy),\a(D))=\frac 12\sum_{j\ge 1}(\e_j M(\l_j)-M(c_j))
\end{equation}

\end{corollary}

\begin{proof}
By \eqref{25jan3} we have that an arbitrary Blaschke product in the product representation (1.20) has this property. 

Recall that the function $R(\l)$, see (2.7), was defined as
$$
R(\l)=R_{D}(\l)=C\frac{i}{\sqrt{\l}}e^{-\int_0^\infty\left\{\frac{1}{\xi-\l}-\frac{\xi}{1+\xi^2}\right\}\ff(\xi)\dd\xi}
$$
where
$$
2 \ff(\xi)=\begin{cases}
-1, & \xi\in (a_j,\l_j)\\
1, & \xi\in(\l_j,b_j)\\
0,& \text{otherwise}
\end{cases}
$$
This decomposition is valid for an arbitrary set of gaps and positions of $\l_j\in[a_j,b_j]$.  We wish to consider specifically the exponential part of $R(\l)$.  Under the assumption \eqref{25jan11} we can specify $C$ by the normalization condition 
$$
\lim_{\mu\to\infty} (-2i\mu) R_D(\mu^2)= 1,
$$
where $\mu$ goes to infinity along imaginary axis. Moreover, in this case for the integral 
$$
\mu\log\left(\frac{2\mu} i R_D(\mu^2)\right)=\int_{\tau^2\in\bbR_+\setminus E}\frac{\tau\mu}{\tau^2-\mu^2} \dd\tau
$$
we have
\begin{equation}\label{25jan7}
\lim_{\mu\to\infty}\mu\log\left(\frac{2\mu} i R_D(\mu^2)\right)=0
\end{equation}
due to the fact that
$
\frac{2\tau\Im\mu}{\tau^2+\Im\mu^2}\le 1
$.

As a combination of \eqref{25jan3} and \eqref{25jan7} we get \eqref{25jan12}. Moreover the corresponding leading term depends only on the Blaschke factor component $\Phi_D(\l)$.
\end{proof}

Similarly we get
\begin{corollary}\label{p30jan4}
Conditions (PW$_M$) and ($k$-GLC) imply
\begin{equation}\label{30jan4}
\lim_{y\to\infty}(\t^{(k)}_{\cR}(p(iy))-(iy)^{2k+1})=0.
\end{equation}

\end{corollary}
\begin{proof}
For $\l<-1$ we have
$$
M^{(k)}(\l)-\Im \l^{k+\frac 1 2 }=\frac{G(\l,-1)}{\pi}\int_{\bbR_+\setminus E}\frac{G(\xi,\l)}{G(-1,\l)}d\xi^{k+\frac 1 2}
\le \frac{G(\l,-1)}{\pi}\int_{\bbR_+\setminus E}\frac{M(\xi)d\xi^{k+\frac 1 2}}{M(-1)}
$$
Since $\sup_{\xi\in \bbR_+\setminus E}M(\xi)<\infty$, we get \eqref{30jan4}.
\end{proof}

\begin{remark}
Recall that the complex Martin function $\T(\l)$ is closely related to  the \textit{integral density of states} $\rho(\xi)$, which can be defined by  
\begin{equation}\label{28jan1}
\T(\l_0)=\int_0^{\l_0}\int_E\frac{\rho(\xi)}{\xi-\l}\frac{\dd\xi}{\sqrt{\xi}} d\l.
\end{equation}
That is, $\rho(\xi)$ for almost every $\xi\in E$ can be defined via  the boundary values of $\T(\l)$, $\l=\xi+i\e$ and
 $$
 \rho(\xi)=\lim_{\e\to 0}\prod_{j\ge 1}\frac{1-c_j/\l}{\sqrt{(1-a_j/\l)(1-b_j/\l)}}
 $$
 In this term for $\fw_M$ we get
\begin{align}\label{eq:qwtoentrpy}
 \fw_M =\sum_{j\ge 1}\left(\int_{c_j}^{b_j}\dd\l\int_E\frac{\rho(\xi) \dd\xi}{\l-\xi}-
 \int_{a_j}^{c_j}\dd\l\int_E\frac{\rho(\xi) \dd\xi}{\l-\xi} \right)=-2\int_E \log\rho(\xi)\frac{\rho(\xi)\dd\xi}{\sqrt{\xi}}.
\end{align}
\end{remark}

\subsection{The KdV hierarchy}

In this subsection we follow basically Dubrovin et al. \cite{DKN}, but to be consistent in the current approach, we avoid the use of symplectic structure.

We define
$$
\chi(p,\a)=\frac{e_{\cR}(p,\a+\fj)}{e_{\cR}(p,\a)}
$$
Its power series expansion at infinity in the local coordinate $1/\mu$ generates the following collection of functions on $\G^*$,
\begin{equation}\label{no1}
\chi(p(\mu),\a)=1+\frac{\chi_0(\a)}{\mu}+\dots+\frac{\chi_{2k}(\a)}{\mu^{2k+1}}+\dots
\end{equation}
We will justify this expansion in the next subsection.
Recall that in this notations
\begin{equation}\label{no2}
(\t_{\cR}(p)+i\pd_\eta)e_{\cR}(p,\a)=\mu_{\cR}(p) e_{\cR}(p,\a+\fj)-\chi_0(\a)e_{\cR}(p,\a)
\end{equation}
Let $\bs_2$ be the shift operator in the complex Euclidian space $\bbC^{2k+2}$, i.e., $\bs_2\delta_n=\d_{n+1}$, $n\le 2k$, where $\{\d_n\}_{n=0}^{2k+1}$ is the standard basis in this space. Then the expansion \eqref{no1} generates the matrix
$$
\chi(\a)=I+\chi_0(\a)\bs_2+\dots+\chi_{2k}(\a) \bs_2^{2k+1}.
$$
We will use the even-odd decomposition of this space $\bbC^{2k+2}=\bbC^{k+1}\oplus \bbC^{k+1}$. In particular, with respect to this  decomposition
$$
\bs_{2}=\begin{bmatrix}0&\bs\\
I&0
\end{bmatrix},\quad
\chi(\a)=
\begin{bmatrix}\chi_o(\a)& \bs\chi_e(\a)\\
\chi_e(\a)& \chi_o(\a)
\end{bmatrix},
$$
where $\bs$ is the standard shift in $\bbC^{k+1}$ and
$$
\chi_o(\a)=
\begin{bmatrix}
1 & & & \\
\chi_1(\a)&1& & \\
 \ddots& \ddots& \ddots&  \\
 \chi_{2k-1}(\a)&\ddots&\chi_1(\a)&1\\
\end{bmatrix},
\
\
\chi_e(\a)=
\begin{bmatrix}
\chi_0(\a) & & & \\
\chi_2(\a)&\chi_0(\a)& & \\
 \ddots& \ddots& \ddots&  \\
 \chi_{2k}(\a)&\ddots&\chi_2(\a)&\chi_0(\a)\\
\end{bmatrix}.
$$

\begin{lemma} Assume that for the generalized Abelian integral $\t_{\cR}^{(k)}(p)$ we have \eqref{p17oct}.
Then the vectors $A(\a), B(\a)\in \bbC^{k+1}$ formed by the coefficients of these polynomials 
are given by
\begin{equation}\label{15nov1}
A(\a)=\chi_o(\a)^{-1}\d_0,\quad B(\a)=\chi_e(\a) \chi_o(\a)^{-1}\d_0.
\end{equation}

\end{lemma}

\begin{proof} Let  $\a(\xi)=\a-\eta^{(k)}\xi$.
We claim   that for an arbitrary $t_k$ the following integral tends to zero
\begin{equation}\label{pe17octc}
\lim_{\mu\to\infty}\int_0^{t_k}
\left(A_k(\mu^2,\a(\xi)) \frac{m_+^{\a(\xi)}(\mu^2)-m_+^{\a(\xi)}(0)}{i}-\mu^{2k+1}-B_k(\mu^2,\a(\xi))\right)d\xi= 0
\end{equation}
where $\mu$ goes to infinity  along the imaginary axis.
Indeed, by \eqref{p17oct} this integral is equal to 
\begin{equation*}\label{pe17octcc}
\left(\t^{(k)}_{\cR}(p(iy))-(iy)^{2k+1}\right)t_k-i\log\frac{e_{\cR}(p(iy),\a(t_k))}{e_{\cR}(p(iy),\a)}.
\end{equation*}
By Corollaries \ref{p30jan1} and \ref{p30jan4} we have \eqref{pe17octc}.

Thus, the polynomial part of the relation \eqref{p17oct} is of the form
$$
\begin{bmatrix} 1\\
B_0(\a)\\
0\\
B_1(\a)\\
\vdots\\
0\\
B_k(\a)
\end{bmatrix}=
\begin{bmatrix}1&  & & & &\\
\chi_0&1 & & & &\\
\chi_1&\chi_0&1 &  & &\\
\chi_2&\chi_1&\chi_0&1& &\\
\ddots &\ddots &\ddots &\ddots &\ddots &\\
\chi_{2k-1}& \dots&\chi_2&\chi_1&\chi_0&1&\\
\chi_{2k}&\dots & \dots&\chi_2&\chi_1&\chi_0&1
\end{bmatrix}(\a)
\begin{bmatrix} 1\\
0\\
\cA_1(\a)\\
0\\
\vdots\\
\cA_k(\a)\\
0
\end{bmatrix}.
$$
Separating odd and even parts, we have
$$
\begin{bmatrix} \d_0\\ B(\a)
\end{bmatrix}=\begin{bmatrix}\chi_o&\bs\chi_e \\
\chi_e&\chi_o
\end{bmatrix}(\a)
\begin{bmatrix} A(\a)\\ 0
\end{bmatrix},
$$
which corresponds to \eqref{15nov1}.
\end{proof}

\begin{lemma} Let $\chi'(\a)=i\pd_{\eta}\chi(\a)$. Then
\begin{equation}
\label{13nov1}
-\chi_o'(\a)=2\chi_e(\a)\chi_o(\a)-2\chi_0(\a)\chi_o(\a),
\end{equation}
\begin{equation}
\label{16nov1}
-\bs\chi_e'(\a)=\chi_o(\a)^2+s\chi_e(\a)^2-I-2\chi_0(\a)s\chi_e(\a).
\end{equation}
\end{lemma}

\begin{proof}
By \eqref{no2}
we have
$$
{-}i\pd_\eta\log \frac{e_{\cR}(p,\a+\fj)}{e_{\cR}(p,\a)}=\mu_{\cR}(p)\frac{e_{\cR}(p,\a+\fj)}{e_{\cR}(p,\a)}-
\mu_{\cR}(p)\frac{e_{\cR}(p,\a)}{e_{\cR}(p,\a+\fj)}-2\chi_0(\a),
$$
which corresponds to the matrix identity
$$
 - \chi(\a)^{-1}\bs_{2k+2}\chi'(\a)=\chi(\a)-\chi(\a+\fj)-2\chi_0(\a)\bs_{2k+2}
$$
or
$$
 -\bs_{2k+2}\chi'(\a)=\chi(\a)^2-I_{2k+2}-2\chi_0(\a)\bs_{2k+2}\chi(\a).
$$
Since in the even-odd decomposition we get
$$
\bs_{2k+2}\chi'(\a)=\begin{bmatrix} \bs\chi'_e& \bs\chi'_o \\
\chi'_o&\chi'_e
\end{bmatrix}(\a),\quad
\chi(\a)^2=
\begin{bmatrix}\chi_o^2+s\chi_e^2& 2\bs\chi_o\chi_e\\
2\chi_o\chi_e& \chi_o^2+\bs\chi_e^2
\end{bmatrix}(\a), 
$$
we obtain \eqref{13nov1}  and \eqref{16nov1}.
\end{proof}

\begin{lemma}\label{16novl1} 
Recall that  
\begin{equation}\label{13nov3}
A_k(\l,\a)=\begin{bmatrix}\l^{k}&\l^{k-1}&\dots&1
\end{bmatrix}\chi_o(\a)^{-1}\delta_0=\sum_{n=0}^k\lambda^{k-n}\cA_n(\a).
\end{equation}
In this notation
\begin{align}\label{14nov1}
(\t_{\cR}^{(k)}(p)+i\pd_{\eta^{(k)}})e_{\cR}(p,\a) 
= (\t_{\cR}(p)+i \pd_{\eta})A_k(\l_{\cR}(p),\a)e_{\cR}(p,\a) \nonumber\\
-\frac 3 2\{i\pd_\eta A_k(\l_{\cR}(p),\a)\}e_{\cR}(p,\a).
\end{align}
\end{lemma}

\begin{proof}
By \eqref{13nov1} we have $2\chi_o(\a)^{-1}\chi_e(\a)\delta_0=2\chi_0(\a)\chi_o(\a)^{-1}\d_0+(\chi_o(\a)^{-1}\d_0)'$.
 Therefore, by \eqref{15nov1} and definition \eqref{13nov3}, we obtain
\begin{equation}\label{13nov5}
 -B_k(\l_{\cR}(p),\a)e_{\cR}(p,\a)=-\chi_0(\a)A_k(\l_{\cR},\a)e_{\cR}(p,\a)-\frac 1 2
\{i\pd_{\eta}A_k(\l_{\cR},\a)\}e_{\cR}(p,\a).
\end{equation}

Using once again \eqref{15nov1}, this time for $A(\a)$, and \eqref{no2}, we have
\begin{align*}
 A_k(\l_{\cR}(p),\a)\mu_{\cR}(p)e_{\cR}(p,\a+\fj)-B_k(\l_{\cR}(p),\a)e_{\cR}(p,\a)
\\
=-\frac 1 2 \{i \pd_{\eta}\{A_k(\l_{\cR}(p),\a)\}e_{\cR}(p,\a)
 +\mu_{\cR}(p)A_k(\l_{\cR}(p),\a)e_{\cR}(p,\a+\fj)
 \\
 -\chi_0(\a)A_k(\l_{\cR}(p),\a)e_{\cR}(p,\a)
\\
=-\frac 1 2
\{i\pd_\eta A_k(\l_{\cR}(p),\a)\} e_{\cR}(p,\a)+A_k(\l_{\cR}(p),\a)(\t_{\cR}(p)+i\pd_\eta)e_{\cR}(p,\a).
\end{align*}
Thus, \eqref{p17oct} means \eqref{14nov1}.
\end{proof}

\begin{theorem}
For $\a\in\G^*$, let $P^{(k)}_\a$, be the following differential operator (w.r.t. $x$) 
$$
P^{(k)}_{\a}:=i\sum_{n=0}^k L_\a^{k-n}\left(A^\a_n(x)\pd_x+\frac 3 2 (A_n^\a)'(x)\right),
$$
where
$$
V^\a(x)=\cV(\a-\eta x),\quad 
A_n^\a(x)=\cA_n(\a-\eta x), \quad L_\a=-\pd_x^2+V^\a(x).
$$
Let
$$
f(p,\a)=\int_{0}^\infty e_{\cR}(p,\a-\eta x)e^{i\t_{\cR}(p)x}\hat f(x,\a) \dd x.
$$
Then
$$
(\t_{\cR}^{(k)}(p)+i\pd_{\eta^{(k)}})f(p,\a)=\int_{0}^\infty e_{\cR}(p,\a-\eta x)e^{i\t_{\cR}(p)x}
\{i\pd_{\eta^{(k)}}+P_k^\a\}\hat f(x,\a) \dd x.
$$
\end{theorem}

\begin{proof}
We have
\begin{align*}
(\t_{\cR}^{(k)}(p)+i\pd_{\eta^{(k)}})f(p,\a)=
\int_{0}^\infty e_{\cR}(p,\a-\eta x)e^{i\t_{\cR}(p)x}\{i\pd_{\eta^{(k)}}\hat f(x,\a)\} \dd x\\
+\int_{0}^\infty \{(\t_{\cR}^{(k)}(p)+i\pd_{\eta^{(k)}})e_{\cR}(p,\a-\eta x)\}e^{i\t_{\cR}(p)x}\hat f(x,\a) \dd x.
\end{align*}
The computations for the second term  are based on Lemma \ref{16novl1}, see \eqref{14nov1}. First,  we note that
\begin{align*}
-i\pd_x A_k(\l_{\cR}(p),\a-\eta x) e_{\cR}(p,\a-\eta x) e^{i\t_{\cR}(p)x}
\\
=
(\t_{\cR}(p)+i\pd_\eta) A_k(\l_{\cR}(p),\a-\eta x) e_{\cR}(p,\a-\eta x) e^{i\t_{\cR}(p)x}.
\end{align*}
Therefore
\begin{align*}
\int_0^\infty\{(\t_{\cR}(p)+i\pd_{\eta})A_k(\l_{\cR},\a-\eta_x) e_{\cR}(p,\a-\eta x)\}e^{i\t_{\cR}(p)x}\hat f(x,\a)\dd x
\\
=\int_0^\infty\{A_k(\l_{\cR},\a-\eta_x) e_{\cR}(p,\a-\eta x)\}e^{i\t_{\cR}(p)x}i\pd_{x}\hat f(x,\a)\dd x.
\end{align*}
On the other hand, for a function  $\hat g(x,\a)$ in the domain, we have
\begin{align*}
\int_{0}^\infty \l_{\cR}(p)^n \cA_n(\a-\eta x)e_{\cR}(p,\a-\eta x)e^{i\t_{\cR}(p)x}\hat g(x,\a) \dd x
\\
=\int_{0}^\infty \l_{\cR}(p)^n e_{\cR}(p,\a-\eta x)e^{i\t_{\cR}(p)x}\{A^\a_n(x)\hat g(x,\a)\} \dd x
\\
=\int_{0}^\infty \{L_\a^n e_{\cR}(p,\a-\eta x)e^{i\t_{\cR}(p)x}\}\{A^\a_n(x)\hat g(x,\a)\} \dd x
\\
=\int_{0}^\infty e_{\cR}(p,\a-\eta x)e^{i\t_{\cR}(p)x}\{L_{\a}^n A^\a_n(x)\hat g(x,\a)\} \dd x.
\end{align*}

Finally, we note that
$
-i\pd_x A_n^\a(x)=i\pd_\eta\cA_n(\a-\eta_x),
$
and, basically, repeat the above computations
\begin{align*}
\int_{0}^\infty \l_{\cR}(p)^n \{i\pd_\eta\cA_n(\a-\eta x)\}e_{\cR}(p,\a-\eta x)e^{i\t_{\cR}(p)x}\hat f(x,\a) \dd x
\\
=\int_{0}^\infty \l_{\cR}(p)^n e_{\cR}(p,\a-\eta x)e^{i\t_{\cR}(p)x}\{-i\pd_x A^\a_n(x)\}\hat f(x,\a) \dd x
\\
=\int_{0}^\infty \{L_\a^n e_{\cR}(p,\a-\eta x)e^{i\t_{\cR}(p)x}\}\{-i\pd_xA^\a_n(x)\}\hat f(x,\a) \dd x
\\
=\int_{0}^\infty e_{\cR}(p,\a-\eta x)e^{i\t_{\cR}(p)x} L_{\a}^n \{-i\pd_x A^\a_n(x)\}\hat f(x,\a) \dd x.
\end{align*}

Combining these three remarks, we obtain
\begin{align*}
\int_0^\infty\{(\t_{\cR}(p)+i\pd_{\eta})A_k(\l_{\cR},\a-\eta x) e_{\cR}(p,\a-\eta x)\}e^{i\t_{\cR}(p)x}\hat f(x,\a)\dd x
\\
-\frac 3 2\int_0^\infty\{i\pd_\eta A_k(\l_{\cR}(p),\a)\}e_{\cR}(p,\a)e^{i\t_{\cR}(p)x}\hat f(x,\a)\dd x
\\
=\int_0^\infty e_{\cR}(p,\a-\eta x)e^{i\t_{\cR}(p)x}P_k^\a\hat f(x,\a)\dd x.
\end{align*}

\end{proof}

\begin{remark}
We note that $P_k^\a$ is self-adjoint in $L^2$ on the whole axis, and therefore can be rewritten into the form
$$
(P_k^{\a})^*=\sum_{n=0}^k \left(i\pd_x A^\a_n(x)-\frac {3i} 2 (A_n^\a)'(x)\right)L_\a^{k-n}=i\sum_{n=0}^k \left(A^\a_n(x)\pd_x-\frac 1 2 (A_n^\a)'(x)\right)L_\a^{k-n}.
$$
\end{remark}

\subsection{The $k$-th gap length condition}

Recall that the function $m_+^\a$ is given as, see (2.9)
$$
m_+^{\a(D)}(\l)=-\frac 1{2R_D(\l)}-\sum_{j\ge 1}\frac{\s^D_j\e_j}{2(1-\l_j/\l)}+m_+^{\a(D)}(0)
$$
where
$$
\s_k^D=2\sqrt{{-\l_k(1-a_k/\l_k)(1-b_k/\l_k)}}\prod_{j\not=k}\frac{\sqrt{(1-a_j/\l_k)(1-b_j/\l_k)}}{1-\l_j/\l_k},
$$
and in the asymptotic expansion we have $\chi_0(\a)=im_+^\a(0)$ and
$$
m_+^\a(\mu^2)=i\mu+i\frac{\chi_1(\a)}{\mu}+i\frac{\chi_2(\a)}{\mu^2}+i\frac{\chi_3(\a)}{\mu^3}+\dots
$$
That is, for the even terms we get
\begin{equation}\label{27jan1}
\chi_{2k}(\a(D))=\frac i 2\sum_{j\ge 1}\s_j^D\l_j^{k}, \quad D=\{(\l_j,\e_j)\}.
\end{equation}
Note, that we have shown continuity of $\chi_0(\a)$ in a quite fashionable way, see Section 4.2.

The situation with 
the odd terms is much simpler. They can be given in terms
of the function
$$
-\frac1 {2i\mu R_D(\l)}=e^{\int {\frac{\ff^D(\xi)\dd\xi}{\xi-\l}} }=1+\frac{\chi_1(\a(D))}{\l}+\frac{\chi_3(\a(D))}{\l^2}+\dots,
$$
where
$$
2\ff^D(\xi)= \begin{cases}
-1, & \xi\in (a_j,\l_j)\\
1, & \xi\in(\l_j,b_j)\\
0,& \text{otherwise}
\end{cases}
$$
Thus $\chi_{2k+1}(\a(D))$ are given as the \textit{standard polynomials} of the moments
$$
\tau_m(\a(D))=\int \xi^m \ff^D(\xi) \dd\xi
$$
Under $\left(k+\frac 1 2\right)$-GLC condition the right hand sides are continuous functions in $D\in\cD(E)$. Therefore  
$\tau_m(\a)$, and respectively $\chi_{2m+1}(\a)$ are continuous in $\a\in\G^*$ for all $m\le k$.

In this section we prove that under the condition ($k$-GLC) the coefficient $\chi_{2k}(\a)$ is well-defined as the corresponding term in asymptotics for $m_+^\a(\l)$, that is, 
\eqref{27jan1} has sense.
\begin{lemma}
Assume that \eqref{21nov0} holds.
To an arbitrary divisor $D$ we associate $\sigma_j=\s_j^D$ by \eqref{eq:intmplus2}.
Then
\begin{equation}\label{29jan2}
\sup_{D\in\cD(E)}\sum_{j\ge 1}\sigma_j\l_j^k<\infty.
\end{equation}
That is, the moment $\chi_{2k}(\a)$ is well defined as a bounded function on $\G^*$.
\end{lemma}

\begin{proof}
We define a Nevanlinna class function $Z(\mu)$ by its argument $\pi\k(\xi)$ on the real axis, where
$$
\k(\xi)=
\begin{cases}
0,&\xi\in\pm\sqrt{E}:=\{\xi:\ \xi^2\in E\}\\
0,& \xi\in(\sqrt{a_j},\sqrt{\l_j})\\
1& \xi\in (\sqrt{\l_j},\sqrt{b_j})\\
0,& \xi\in (-\sqrt{b_j},-\sqrt{\l_j})\\
1,& \xi\in (-\sqrt{\l_j},-\sqrt{a_j})
\end{cases}
$$
To be more precise, with   a suitable choice of the positive multiplier, we have
\begin{equation}\label{29jan0}
Z(\mu)=\prod_{j\ge 1}\frac{\mu-\sqrt{b_j}}{\mu-\sqrt{\l_j}}\frac{\mu+\sqrt{a_j}}{\mu+\sqrt{\l_j}}= 
e^{\int_{\bbR\setminus\pm\sqrt{E}}\k(\xi)\frac{d\xi}{\xi-\mu}}
\end{equation}
Since
$$
\int_{\bbR\setminus\pm\sqrt{E}}\xi^{2k}\k(\xi){\dd\xi} =
\sum_{j\ge 1}\frac{b_j^{k+\frac 1 2}-a_j^{k+\frac 1 2}}{2k+1},
$$
by \eqref{21nov0},  the function $Z(\mu)$ has a \textit{real} asymptotic expansion at infinity (along the imaginary axis) up to the term 
$\mu^{-(2k+1)}$, i.e.,
\begin{equation}\label{29jan1}
\lim_{\mu\to\infty}\mu^{2k+1}\{Z(\mu)-1+\frac{\tau_0}{\mu}+\dots+\frac{\tau_{2k-1}}{\mu^{2k}}\}=-\tau_{2k}.
\end{equation}

An additive representation of this function is of the form
$$
Z(\mu)=1+\sum_{j\ge 1}\left(\frac{\rho_j^+}{\sqrt{\l_j}-\mu}-\frac{\rho_j^-}{\sqrt{\l_j}+\mu}\right),
$$
where
$$
\rho_n^+=(\sqrt{b_n}-\sqrt{\l_n})\left|\prod_{j\not=n}\frac{\sqrt{\l_n}-\sqrt{b_j}}{\sqrt{\l_n}-\sqrt{\l_j}}\frac{\sqrt{\l_n}+\sqrt{a_j}}{\sqrt{\l_n}+\sqrt{\l_j}}
\right|
$$
and
$$
\rho_n^-=(\sqrt{\l_n}-\sqrt{a_n})\left|\prod_{j\not=n}\frac{\sqrt{\l_n}+\sqrt{b_j}}{\sqrt{\l_n}+\sqrt{\l_j}}\frac{\sqrt{\l_n}-\sqrt{a_j}}{\sqrt{\l_n}-\sqrt{\l_j}}
\right|.
$$

Note that
$$
\rho_n^+\rho_n^-=\frac{({b_n}-{\l_n})({\l_n}-{a_n})}{\l_n(1+\sqrt{b_n/\l_n})(1+\sqrt{a_n/\l_n})}
\left|\prod_{j\not=n}\frac{{\l_n}-{b_j}}{{\l_n}-{\l_j}}\frac{{\l_n}-{a_j}}{\l_n-{\l_j}}
\right|
$$
We fix $C$ such that $b_n-a_n\le 1/2$ for $a_n\ge C$. Then
$$
\sqrt{b_n/\l_n}\le b_n/\l_n\le b_n/a_n\le 1/(2a_n)+1\le 1/(2C)+1.
$$
For such values we get
$$
\frac C{1+4C}\le \frac 1{(1+\sqrt{b_n/\l_n})(1+\sqrt{a_n/\l_n})}.
$$
That is
$$
\frac C{1+4C}\sigma_n\le 2\sqrt{\rho_n^+\rho_n^-}\le \rho_n^++\rho_n^-.
$$
Due to the asymptotic expansion \eqref{29jan1} for $Z(\mu)$, see \eqref{29jan0},
$$
\sup_{D\in\cD(E)}\sum_{j\ge 1}(\rho_j^++\rho_j^-)\l_j^k<\infty.
$$
On  the other hand $\chi_0(\a)$ is continuous on $\G^*$, that is, 
$$
\sup_{D\in\cD(E)}\sum \s_j^D=\sup_{\a\in\G^*}\chi_0(\a)<\infty.
$$
Thus we have \eqref{29jan2} by
$$
\sum_{j\ge 1}{\sigma_j\lambda_j^k}\le 
\sum_{\l_j\le C}{\sigma_j\lambda_j^k}+
\sum_{\l_j\ge a_j\ge C}{\sigma_j\lambda_j^k}\le
C^k \sum_{j\ge 1}\s_j+\frac{C}{1+4C}\sum_{j\ge 1}(\rho_j^++\rho_j^-)\l_j^k.
$$
\end{proof}

\bibliographystyle{amsplain}
\providecommand{\MR}[1]{}
\providecommand{\bysame}{\leavevmode\hbox to3em{\hrulefill}\thinspace}
\providecommand{\MR}{\relax\ifhmode\unskip\space\fi MR }
\providecommand{\MRhref}[2]{%
  \href{http://www.ams.org/mathscinet-getitem?mr=#1}{#2}
}
\providecommand{\href}[2]{#2}

\end{document}